\documentclass[11pt]{amsart}

\usepackage[lmargin=1in,rmargin=1in,tmargin=1in,bmargin=1in]{geometry}
\RequirePackage{amsmath} 
\RequirePackage{amssymb}
\usepackage{amscd,latexsym,amsthm,amsfonts,amssymb,amsmath,amsxtra}
\usepackage[colorlinks=true,urlcolor=blue,citecolor=blue]{hyperref}
\usepackage{color}
\usepackage[all]{xy}
\usepackage[OT2,T1]{fontenc}
\usepackage{bm}
\usepackage{mathtools}
\usepackage{mathrsfs}
\usepackage{xcolor}
\usepackage{comment}
\usepackage{marginnote}
\usepackage{enumitem}
\usepackage{thmtools, thm-restate}

\DeclareSymbolFont{cyrletters}{OT2}{wncyr}{m}{n}
\DeclareMathSymbol{\Sha}{\mathalpha}{cyrletters}{"58}

\let\Re\undefined

\DeclareMathOperator{\Re}{Re}

\DeclareMathOperator{\ord}{ord}
\DeclareMathOperator{\supp}{supp}

\DeclareMathOperator{\GL}{GL}

\DeclareMathOperator{\sign}{sign}

\newcommand{\bF}{\mathbb{F}}
\newcommand{\bQ}{\mathbb{Q}}
\newcommand{\bA}{\mathbb{A}}
\newcommand{\cF}{\mathcal{F}}
\newcommand{\ve}{\varepsilon}

\newcommand{\cM}{\mathcal{M}}
\newcommand{\bR}{\mathbb{R}}
\newcommand{\bC}{\mathbb{C}}
\newcommand{\cO}{\mathcal{O}}
\newcommand{\fa}{\mathfrak{a}}
\newcommand{\fb}{\mathfrak{b}}
\newcommand{\fc}{\mathfrak{c}}
\newcommand{\fn}{\mathfrak{n}}
\newcommand{\ff}{\mathfrak{f}}
\newcommand{\fl}{\mathfrak{l}}

\newcommand{\fp}{\mathfrak{p}}
\newcommand{\fm}{\mathfrak{m}}

\newcommand{\fq}{\mathfrak{q}}

\newcommand{\fd}{\mathfrak{d}}

\newcommand{\fN}{\mathfrak{N}}
\newcommand{\fr}{\mathfrak{r}}

\newcommand{\cE}{\mathcal{E}}
\newcommand{\cL}{\mathcal{L}}

\newcommand{\tk}{\mathbf{k}}

\def\Re{\operatorname{Re}}

\newcommand{\Mod}[1]{\ (\mathrm{mod}\ #1)}
        \newcommand{\sym}{\operatorname{sym}}

	\newcommand{\Res}{\operatorname{Res}}

	\newcommand{\sgn}{\operatorname{sgn}}

	\newcommand{\Ad}{\operatorname{Ad}}

	\newcommand{\RNum}[1]{\uppercase\expandafter{\romannumeral #1\relax}}

\begin{document}
\theoremstyle{plain}
	\newtheorem{thm}{Theorem}[section]
	\newtheorem{cor}[thm]{Corollary}
	\newtheorem{thmy}{Theorem}
        \newtheorem{conjy}{Conjecture}
        \newtheorem{cory}{Corollary}
	\renewcommand{\thethmy}{\Alph{thmy}}
    \renewcommand{\theconjy}{\Alph{conjy}}
    \renewcommand{\thecory}{\Alph{cory}}
    
    \newenvironment{thmx}{\stepcounter{thm}\begin{thmy}}{\end{thmy}}

    \newenvironment{conjx}{\stepcounter{conj}\begin{conjy}}{\end{conjy}}
    
    \newenvironment{corx}{\stepcounter{cor}\begin{cory}}{\end{cory}}

	\newtheorem{hy}[thm]{Hypothesis}
	\newtheorem*{thma}{Theorem A}
	\newtheorem*{corb}{Corollary B}
	\newtheorem*{thmc}{Theorem C}
        \newtheorem*{thmd}{Theorem D}
	\newtheorem{lemma}[thm]{Lemma}  
	\newtheorem{prop}[thm]{Proposition}
	\newtheorem{conj}[thm]{Conjecture}
	\newtheorem{fact}[thm]{Fact}
	\newtheorem{claim}[thm]{Claim}
	\newtheorem{question}[thm]{Question}
	
	\theoremstyle{definition}
	\newtheorem{defn}[thm]{Definition}
	\newtheorem{example}[thm]{Example}
	\theoremstyle{remark}
	\newtheorem{remark}[thm]{Remark}	
	\numberwithin{equation}{section}

\title[]{Low-lying zeros of Hilbert modular $L$-functions weighted by powers of central $L$-values}
\author{Zhining Wei, Liyang Yang and Shifan Zhao}

\address{Kassar House, 151 Thayer St, Providence, RI 02912 USA}
\email{zhining$\_$wei@brown.edu}

\address{253-37 Caltech, Pasadena, CA 91125, USA}
\email{lyyang@caltech.edu}

\address{Math Building, 231 W. 18th Ave, Columbus, OH 43210 USA}
\email{zhao.3326@osu.edu}

\begin{abstract}
Let $\cF(\tk,\fq)$ be the set of primitive Hilbert modular forms of weight $\tk$ and prime level $\fq$, with trivial central character. We study the one-level density of low-lying zeros of $L(s,\pi)$ weighted by powers of central $L$-values $L(1/2,\pi)^r$, where $\pi$ runs through $\cF(\tk,\fq)$. For $r=1,2,3$, we show that the resulting distributions $W_r$ match with predictions from Random Matrix Theory. For general $r \geq 1$, we also formulate a conjectural formula for $W_r$ based on the ``recipe'' in \cite{ConreyFarmerKeatingRubinsteinSnaith2005}.
\end{abstract}
	
\maketitle

\tableofcontents

\section{Introduction}

\subsection{Background}
It is widely believed that there is a close connection between zeros of $L$-functions and eigenvalues of random matrices. Evidence of such a connection was first found for the Riemann zeta function $\zeta(s)$ by Montgomery \cite{Montgomery1973}, who investigated the pair correlation of non-trivial zeros of $\zeta(s)$ and noticed that it is the same as the pair correlation of eigenvalues of Hermitian random matrices. Similar phenomenon was later discovered by Rudnick and Sarnak \cite{RudnickSarnak1996} for general automorphic $L$-functions on $\GL(m)$. Their results suggest that the $n$-level correlation of zeros of any individual principal $L$-function $L(s,\pi)$ attached to a cuspidal automorphic representation $\pi$ of $\GL_m(\bA_\bQ)$ follows a \textit{universal} gaussian unitary ensemble (GUE) model.

Unlike the $n$-level correlation statistic for an \textit{individual} $L$-function, there is another statistic, called the $n$-level density of low-lying zeros, that is defined for \textit{families} of $L$-functions and is sensitive to families. Katz and Sarnak \cite{KatzSarnak1999}\cite{KatzSarnak19992} studied the $n$-level density of low-lying zeros of zeta functions associated to varieties over function fields $\bF_q(t)$. For these they found a spectral interpretation in terms of eigenvalues of Frobenius maps on cohomology groups. They proposed a conjecture, called the Density Conjecture, predicting that low-lying zeros of $L$-functions in a family behave like low-lying eigenvalues of random matrices of a specific symmetry type.

Let us describe the Density Conjecture in general terms. Let $\cF$ be a finite set of automorphic representations. Let $\Phi:\bR \to \bR$ be a Schwartz function whose Fourier transform $\hat{\Phi}$ has compact support. We call $\Phi$ a \textit{test function} throughout this paper. Assume the Generalized Riemann Hypothesis (GRH) for $L(s,f)$, where $f \in \cF$. Define the \textit{1-level density of low-lying zeros} of $L(s,f)$ associated to $\Phi$ to be
\begin{equation}\label{1-level density}
D(f;\Phi) := \sum_{\rho_f}\Phi\left(\frac{\gamma_f}{2\pi}\log c_f\right),
\end{equation}
where the sum is over non-trivial zeros $\rho_f = 1/2+i\gamma_f$ of $L(s,f)$, and $c_f$ is the analytic conductor of $L(s,f)$. Regarding the average behavior of $D(f;\Phi)$, Katz and Sarnak proposed the following \textit{Density Conjecture} \cite{KatzSarnak1999}\cite{KatzSarnak19992}:

\begin{conj}[Density Conjecture]\label{density conjecture}
    There exists a distribution $W_\cF$ on $\bR$ depending on $\cF$, such that for any test function $\Phi$, we have 
    $$\lim_{|\cF| \to \infty} \frac{1}{|\cF|} \sum_{f \in \cF} D(f;\Phi) = \int_\bR \Phi(x)W_\cF(x) dx.$$
\end{conj}

Results in \cite{KatzSarnak1999} suggest that the distribution $W_\cF$ depends on $\cF$ through a symmetry group $G(\cF)$ in Random Matrix Theory (RMT), and the distribution $W_{G(\cF)}$ is exactly the one predicted by RMT. Thus $G(\cF)$ is called the ``symmetry type'' of $\cF$. For a list of possible symmetry types and their corresponding distributions, see \cite{IwaniecLuoSarnak2000}.

The Density Conjecture \ref{density conjecture} has been studied for many families $\cF$. See \cite{IwaniecLuoSarnak2000,Rubinstein2001,FouvryIwaniec2003,Young2006,ConreySnaith2007,ChoKim2015,ShinTemplier2016,LiuMiller2017,KimWakatsukiYamauchi2020,BaluyotChandeeLi2024}, to name a few. In all results mentioned above, either the support of $\hat{\Phi}$ is restricted to a certain range, or the Ratio Conjecture for the corresponding $L$-functions is assumed. 

Recently, weighted versions of the Density Conjecture \ref{density conjecture} have attracted much attention. Let $\omega_f$ be certain ``weight'' associated to $f \in \cF$. Instead of studying the natural average of $D(f;\Phi)$, we consider the following \textit{weighted} average
$$\lim_{|\cF| \to \infty}\frac{1}{\sum_{f \in \cF}\omega_f}\sum_{f \in \cF} \omega_fD(f;\Phi).$$
See \cite{KowalskiSahaTsimerman2012,KnightlyReno2019,SugiyamaSuriajaya2022,Fazzari2024,BettinFazzari2024,Sugiyama2025,Zhao2025} for results in this direction. In particular, when $\omega_f = L(1/2,f)^r$ is a power of central $L$-value, Fazzari proposed the following Weighted Density Conjecture \cite[Conjecture 2.1]{Fazzari2024}:

\begin{conj}[Weighted Density Conjecture]\label{weighted density conjecture}
    Let $\{L(s,f)\}_{f \in \cF}$ be a family of $L$-functions with symmetry type $G = G(\cF)$, where $G \in \{\text{U},\text{Sp},\text{SO(even)}\}$. Let $r \geq 1$ be a positive integer. Let $V(z) = |z|^2$ if $G = U$, and $V(z) = z$ otherwise. Then there exists a distribution $W_G^r$ on $\bR$ depending on $G$ and $r$, such that for any test function $\Phi$, we have 
    $$\lim_{|\cF| \to \infty} \frac{1}{\sum_{f \in \cF}V(L(1/2,f))^r}\sum_{f \in \cF}V(L(1/2,f))^rD(f;\Phi) = \int_\bR \Phi(x)W_G^r(x) dx.$$
    Moreover, $W_{Sp}^r(x)$ has the following explicit expression
    $$W_{Sp}^r(x) = 1-(2r+1)\frac{\sin(2\pi x)}{2\pi x} + \sum_{j=1}^r \frac{r(r+1)}{2^{2j-2}\pi^{2j-1}} \frac{c_{j,r}}{2j-1} \frac{d^{2j-1}}{dx^{2j-1}}\left(\frac{1-\cos(2\pi x)}{2\pi x}\right),$$
    where $c_{j,r}$ is given by 
    $$c_{j,r} := \frac{1}{j}\binom{r-1}{j-1}\binom{r+j}{j-1} .$$
    The other two distributions  can be recovered from $W_{Sp}^r$ by the following relations
    \begin{align*}
        W_{SO(\text{even})}^r(x) = W_{Sp}^{r-1}(x)\hspace{10mm}\mbox{and}\hspace{10mm}
        W_U^r(x) = \frac{W_{SO(\text{even})}^r(x) + W_{Sp}^r(x)}{2}.
    \end{align*}
\end{conj}

\begin{remark}
    A remarkable point of the above conjecture is that the weighted distributions $W_{G(\cF)}^r$ depend \textit{only} on the symmetry type $G$, and \textit{not} on the underlying family $\cF$. This can be viewed as further evidence of the Katz-Sarnak philosophy. However, as we shall discuss later, this universality is only supported by known results for $r=1$. Our results in this paper provide the first evidence for $r=2,3$.
\end{remark}

\begin{remark}
    Fazzari also considered RMT analogies in \cite{Fazzari2024}, and showed that the conjectural weighted distributions $W_G^r$ are those coming from weighted low-lying eigenvalues of random matrices with symmetry type $G$, for small values of $r$.
\end{remark}

\subsection{Main Results}
In this paper we study low-lying zeros of standard $L$-functions of Hilbert modular forms, weighted by powers of central $L$-values. Note that the unweighted family has orthogonal symmetry type \cite[Theorem 1.1]{IwaniecLuoSarnak2000}\cite[Theorem 1.1]{LiuMiller2017}.

We proceed to describe our main results. Let $F$ be a totally real number field with ring of integers $\cO_F$. We use $v$ to denote a place of $F$. Let $\tk = (k_v)_{v|\infty}$ be a weight vector, where each $k_v \geq 4$ is even. Set $\|\tk\| = \prod_{v|\infty}k_v$. Let $\fq \subset \cO_F$ be either a prime ideal or the whole $\cO_F$. Let $N(\fq) = [\cO_F:\fq]$ be the absolute norm of $\fq$. Let $\cF(\tk,\fq)$ be the set of cuspidal automorphic representations of $\GL_2(\bA_F)$ of weight $\tk$ and exact level $\fq$ with trivial central character. To each $\pi \in \cF(\tk,\fq)$ denote by $L(s,\pi)$ the standard $L$-function attached to $\pi$, and denote by $L(s,\pi,\Ad)$ the adjoint $L$-function of $\pi$. For a test function $\Phi$ define the $1$-level density $D(\pi;\Phi)$ as in (\ref{1-level density}) with $c_\pi = \|\tk\|^2N(\fq)$.

We study the average of $D(\pi;\Phi)$ weighted by $L(1/2,\pi)^r$ for $r=1,2,3$, where $\pi$ runs through $\cF(\tk,\fq)$. For $r=1,2$ we study the \textit{joint} aspect as $\|\tk\|N(\fq) \to \infty$, and for $r=3$ we study the \textit{level} aspect as $N(\fq) \to \infty$. To save notations we set $Q_r := \|\tk\|^2N(\fq)$ if $r=1,2$, and $Q_r := N(\fq)$ if $r=3$. We also set $\alpha_1 = \alpha_3 = 1/2$ and $\alpha_2 = 1/4$. Our first result is as follows.

\begin{thmx}\label{thm. llz, with harmonic weight}
Let $r=1,2,3$. Let $\Phi$ be a test function with $\supp(\hat{\Phi}) \subset (-\alpha_r,\alpha_r)$. Then we have 
$$\lim_{Q_r \to \infty} \frac{1}{\sum_{\pi \in \cF(\tk,\fq)}\frac{L(1/2,\pi)^r}{L(1,\pi,\Ad)}}\sum_{\pi \in \cF(\tk,\fq)}\frac{L(1/2,\pi)^r}{L(1,\pi,\Ad)}D(\pi;\Phi) = \int_\bR\Phi(x)W_r(x)\,dx.$$
Here the distributions $W_r$ are defined by
\begin{flalign*}
W_1(x)&=1-\frac{\sin(2\pi x)}{2\pi x},\\
    W_2(x)&=1+\frac{\sin(2\pi x)}{2\pi x}+\frac{\cos(2\pi x)-1}{\pi^2x^2},\\
    W_3(x)&=1-\frac{\sin(2\pi x)}{2\pi x}-\frac{3(\cos(2\pi x)-1)}{\pi^2x^2}+\frac{6(\sin(2\pi x)-2\pi x)}{\pi^3x^3}+\frac{3(\cos(2\pi x)-1+2\pi^2x^2)}{\pi^4x^4}.
\end{flalign*}
\end{thmx}

\begin{remark}
    Theorem \ref{thm. llz, with harmonic weight} extends \cite[Theorem 1.1]{KnightlyReno2019} (in the case of $\chi$ trivial) from $F = \bQ$ to any totally real number field $F$, and from $r=1$ to $r =1,2,3$.
\end{remark}

The harmonic weight $L(1,\pi,\Ad)^{-1}$ is at present because it naturally appears in the relative trace formula that we applied. We can remove it to get the following \textit{purely} $L(1/2,\pi)^r$-weighted result. 

\begin{thmx}\label{thm. llz, without harmonic weight}
Let $r=1,2,3$. Let $\Phi$ be a test function with $\supp(\hat{\Phi}) \subset (-\alpha_r,\alpha_r)$. Then we have 
$$\lim_{Q_r \to \infty} \frac{1}{\sum_{\pi \in \cF(\tk,\fq)}{L(1/2,\pi)^r}}\sum_{\pi \in \cF(\tk,\fq)}L(1/2,\pi)^rD(\pi;\Phi) = \int_\bR\Phi(x)W_r(x)\,dx.$$
Here the distributions $W_r$ are defined as in Theorem \ref{thm. llz, with harmonic weight}.
\end{thmx}

\begin{remark}
    Note that our distributions $W_r$ match the weighted distributions $W_{\text{SO(even)}}^r$ in Conjecture \ref{weighted density conjecture}, for $r = 1,2,3$. Thus Theorem \ref{thm. llz, without harmonic weight} provides new evidence towards Conjecture \ref{weighted density conjecture}. Moreover, when $r>1$, combined with \cite[Theorem 2.5]{Fazzari2024}, Theorem \ref{thm. llz, without harmonic weight} provides the first evidence that the weighted distributions $W_{G(\cF)}^r$ in Conjecture \ref{weighted density conjecture} depend \textit{only} on the symmetry type $G = G(\cF)$, and not on the underlying family $\cF$. Also note that when $r>1$, Theorem \ref{thm. llz, without harmonic weight} is the first result towards Conjecture \ref{weighted density conjecture} \textit{without} assuming the Ratio Conjecture.
\end{remark}

\begin{remark}
    Although it is not surprising that removing the harmonic weight $L(1,\pi,\Ad)^{-1}$ would not change the symmetry type, Theorem \ref{thm. llz, without harmonic weight} is the first result among all weighted low-lying zeros results that rigorously proves this fact. 
\end{remark}

Let $\fp \subset \cO_F$ be a fixed prime ideal with $\fp \neq \fq$. The proof of Theorem \ref{thm. llz, with harmonic weight} relies on the following weighted equidistribution result of the $\fp^{th}$-Hecke eigenvalues $\lambda_\pi(\fp)$, which has its own interest. Equidistribution results of this type have seen extensive investigation, including works such as \cite{Serre1997,ConreyDukeFarmer1997,ShinTemplier2016,Li2009,RamakrishnanRogawski2005,FeigonWhitehouse2009,SugiyamaTsuzuki2016,KnightlyReno2019,MichelRamakrishnanYang2025}.

\begin{thmx}\label{thm. sato tate, with harmonic weight}
Let $\fp \subset \cO_F$ be a prime ideal with $\fp \neq \fq$. Let $\phi$ be a continuous functions on $\bR$. Let $r=1,2,3$. Then we have  
$$\lim_{Q_r \to \infty} \frac{1}{\sum_{\pi \in \cF(\tk,\fq)}\frac{L(1/2,\pi)^r}{L(1,\pi,\Ad)}}\sum_{\pi \in \cF(\tk,\fq)}\frac{L(1/2,\pi)^r}{L(1,\pi,\Ad)}\phi(\lambda_\pi(\fp)) = \int_\bR\phi(x)\,d\mu_{\fp,r}(x).$$
Here the measure $\mu_{\fp,r}$ is defined by
\[\,d\mu_{\fp,r}(x)=\left(1-\frac{1}{N(\fp)}\right)^{\frac{r(r-1)}{2}}\frac{1}{\left(1-\frac{x}{N(\fp)^{1/2}}+\frac{1}{N(\fp)}\right)^r}\,d\mu_{\infty}(x),\]
where $\mu_\infty$ is the Sato-Tate measure defined by 
$$d\mu_\infty(x) = \begin{cases}
    \frac{1}{\pi}\sqrt{1-\frac{x^2}{4}}dx & -2 \leq x \leq 2 \\
    0 & \text{otherwise}.
\end{cases}$$
\end{thmx}

\begin{remark}
    Theorem \ref{thm. sato tate, with harmonic weight} extends \cite[Theorem 1.5]{KnightlyReno2019} (in the case of $\chi$ trivial) from $F=\bQ$ to any totally real number field $F$, and from $r=1$ to $r=1,2,3$.
\end{remark}

Like many other equidistribution results, removing the harmonic weight $L(1,\pi,\Ad)^{-1}$ would change the resulting measure. Explicitly, we have the following result.

\begin{thmx}\label{thm. sato tate, without harmonic weight}
Let $\fp \subset \cO_F$ be a prime ideal with $\fp \neq \fq$. Let $\phi$ be a continuous functions on $\bR$. Let $r=1,2,3$. Then we have  
$$\lim_{Q_r \to \infty} \frac{1}{\sum_{\pi \in \cF(\tk,\fq)}L(1/2,\pi)^r}\sum_{\pi \in \cF(\tk,\fq)}L(1/2,\pi)^r\phi(\lambda_\pi(\fp)) = \int_\bR\phi(x)\,d\mu_{\fp,r}^{un}(x).$$
Here the measures $\mu_{\fp,r}^{un}$ are defined by 
\begin{flalign*}
\,d\mu_{\fp,1}^{un}(x)&
=(1-N(\fp)^{-1})(1+N(\fp)^{-1})^2\frac{1}{\left(1-\frac{x}{N(\fp)^{1/2}}+\frac{1}{N(\fp)}\right)^2\left(1+\frac{x}{N(\fp)^{1/2}}+\frac{1}{N(\fp)}\right)}\,d\mu_{\infty}(x),\\
\,d\mu_{\fp,2}^{un}(x)&=\frac{(1-N(\fp)^{-2})^3}{(1+N(\fp)^{-2})}\frac{1}{\left(1-\frac{x}{N(\fp)^{1/2}}+\frac{1}{N(\fp)}\right)^3\left(1+\frac{x}{N(\fp)^{1/2}}+\frac{1}{N(\fp)}\right)}\,d\mu_{\infty}(x),\\
\,d\mu_{\fp,3}^{un}(x)&=\frac{(1-N(\fp)^{-1})^5(1+N(\fp)^{-1})^4}{(1+N(\fp)^{-1}+4N(\fp)^{-2}+N(\fp)^{-3}+N(\fp)^{-4})}
\frac{1}{\left(1-\frac{x}{N(\fp)^{1/2}}+\frac{1}{N(\fp)}\right)^4\left(1+\frac{x}{N(\fp)^{1/2}}+\frac{1}{N(\fp)}\right)}\,d\mu_{\infty}(x).
\end{flalign*}
\end{thmx}

\subsection{Main Conjectures}

In proving the main results mentioned above, formulas of the $\lambda_\pi(\fn)$-weighted $r^{th}$-moments of $L(1/2,\pi)$ play a central role. Here $\fn \subset \cO_F$ is an integral ideal coprime to $\fq$. Such moment formulas are established for $r=1,2,3$ in this paper. For \textit{general} $r \geq 1$ it is also possible to apply the ``recipe'' in \cite{ConreyFarmerKeatingRubinsteinSnaith2005} to give conjectural formulas of the weighted $r^{th}$-moments of $L(1/2,\pi)$, and use these formulas to formulate conjectures on weighted low-lying zeros and on weighted equidistribution of $\lambda_\pi(\fp)$.

To state our conjectures let us introduce some more notations. For $r\geq1,$ set $\Delta(a_1,\ldots,a_r)$ to be the vandermonde discriminant
\[\Delta(a_1,\ldots,a_r)=\begin{cases}
	1&\mbox{if $r=1$}\\
	\prod_{1\leq i<j\leq r}(a_j-a_i)&\mbox{if $r\geq2$}.
\end{cases}\]
For $n\geq0,$ define the function $h_n$ to be, (or see \S \ref{subsec. fourier transformation})
\begin{equation}\label{eq. the h function}
 h_n(x) = \int_0^1 t^n \cos(2\pi x t) \, dt.
\end{equation}
For $j \geq 0$ we introduce
\[b_{r}(j)=\frac{(-1)^{r(r-1)/2}2^{r}}{r!(2\pi i)^r}\oint\cdots\oint\frac{\Delta(z_1^2,\ldots,z_r^2)\Delta(z_1,\ldots,z_r)(z_1^{j}+\cdots z_r^j)}{z_1^{2r-1}\cdots z_r^{2r-1}}e^{z_1+\cdots+z_r}\,dz_1\cdots\,dz_r,\]
where the integral is along small circles around $0$.

Using the notations above we formulate the following conjecture on weighted low-lying zeros.
\begin{conjx}\label{conj. arithmetic low lying zero}
	Let $r \geq 1$ be a positive integer. Let $\Phi$ be a test function with $\supp(\hat{\Phi})$ sufficiently small. Then we have 
	$$\lim_{\|\tk\|N(\fq) \to \infty} \frac{1}{\sum_{\pi \in \cF(\tk,\fq)}\frac{L(1/2,\pi)^r}{L(1,\pi,\Ad)}}\sum_{\pi \in \cF(\tk,\fq)}\frac{L(1/2,\pi)^r}{L(1,\pi,\Ad)}D(\pi;\Phi) = \int_\bR\Phi(x)W_r(x)\,dx.$$
	Here the distribution $W_r$ is defined by
	\[W_r(x)=1+h_0(x)-r\sum_{j\geq 0}\frac{(-1)^j2^{j+1}b_{r}(j)}{j!b_{r}(0)}h_j(x).\]
\end{conjx}

\begin{remark}\label{rem. comparison between W and W(SO)}
    In section \S \ref{apx. explicit calculation for bnj}, we will show that $W_r(x)= W_{SO(\text{even})}^r(x)$ for $r=1,2,3,4$.
\end{remark}

For a prime ideal $\fp \subset \cO_F$ with $\fp \neq \fq$ and $r \geq 1$, we set
\begin{equation}\label{eq. local component}
a(\fp,r)=\frac{2}{\pi}\int_0^{\pi}\sin^2\theta\frac{1}{\left(1-\frac{2\cos\theta}{N(\fp)^{1/2}}+\frac{1}{N(\fp)}\right)^r}\,d\theta.
\end{equation}

We then formulate the following conjecture on weighted equidistribution of $\lambda_\pi(\fp)$. 

\begin{conjx}\label{conj. arithmetic Sato-Tate}
Let $\fp \subset \cO_F$ be a prime ideal with $\fp \neq \fq$. Let $\phi$ be a continuous functions on $\bR$. Let $r\geq 1$ be a positive integer. Then we have  
$$\lim_{\|\tk\|N(\fq) \to \infty} \frac{1}{\sum_{\pi \in \cF(\tk,\fq)}\frac{L(1/2,\pi)^r}{L(1,\pi,\Ad)}}\sum_{\pi \in \cF(\tk,\fq)}\frac{L(1/2,\pi)^r}{L(1,\pi,\Ad)}\phi(\lambda_\pi(\fp)) = \int_\bR\phi(x)\,d\mu_{\fp,r}(x).$$
Here the measure $\mu_{\fp,r}$ is given by:
	\[\,d\mu_{\fp,r}(x)=\frac{1}{a(\fp,r)}\frac{1}{\left(1-\frac{x}{N(\fp)^{1/2}}+\frac{1}{N(\fp)}\right)^r}\,d\mu_{\infty}(x).\]
\end{conjx}
\begin{remark}
In \S \ref{apx. arithmetic Sato-Tate}, we will compute $a(\fp,r)$ explicitly, and show that they agree with Theorem \ref{thm. sato tate, with harmonic weight} for $r=1,2,3$.
\end{remark}

\subsection{Paper Outline and Notations}
This paper is organized as follows. In \S \ref{section 2}, we establish the weighted first, second and cubic moment formulas. In \S \ref{sec, weighted sato tate}, we apply the weighted moment formulas to prove the weighted equidistribution results Theorem \ref{thm. sato tate, with harmonic weight} and Theorem \ref{thm. sato tate, without harmonic weight}. In \S \ref{sec weighted low lying zero}, we prove the weighted low-lying zeros results Theorem \ref{thm. llz, with harmonic weight} and Theorem \ref{thm. llz, without harmonic weight}. For simplicity, we will skip the proof for the $r=1$ case in each theorem, as the proof is much easier. In \S \ref{sec. main conjectures and proofs}, we apply the recipe in \cite{ConreyFarmerKeatingRubinsteinSnaith2005} to formulate a weighted moment conjecture \ref{conj.main weighted moment conjecture}. Then we use it to derive Conjecture \ref{conj. arithmetic low lying zero} and Conjecture \ref{conj. arithmetic Sato-Tate}. 

We introduce the following notations that are used throughout this paper. Let $F$ be a totally real number field with degree $d_F$, ring of integers $\cO_F$, absolute discriminant $D_F$, and Dedekind zeta function $\zeta_F(s)$. Let $\gamma_{-1}$ be the residue of $\zeta_F(s)$ at $s=1$, and $\gamma_0$ be the constant term in the Laurent expansion of $\zeta_F(s)$ at $s=1$. We use $v$ to denote a place of $F$. Let $\tk = (k_v)_{v|\infty}$ be a weight vector, where each $k_v \geq 4$ is even. We set $\|\tk\| := \prod_{v|\infty}k_v$. Let $\fq \subset \cO_F$ be either a prime ideal or $\fq = \cO_F$. Let $N(\fq) = [\cO_F:\fq]$ be the absolute norm of $\fq$. Let $\zeta_\fq(s) = (1-N(\fq)^{-s})^{-1}$ be the local zeta function at $\fq$ when $\fq$ is prime, and let $\zeta_{\cO_F}(s) = 1$. We set $\delta_\tk = 1$ if $\sum_{v|\infty}k_v \equiv0 \bmod 4$, and $\delta_{\tk} = 0$ otherwise. We set $\delta_{\fq\subsetneq\cO_F} = 1$ if $\fq$ is prime, and $\delta_{\fq\subsetneq\cO_F} = 0$ if $\fq = \cO_F$. Let $\tau$ be the divisor function over $F$. Thus for an integral ideal $\fn \subset \cO_F$, $\tau(\fn) = \sum_{\fn = \fa \fb} 1$.

Let $\cF(\tk,\fq)$ be the set of cuspidal automorphic representations of $\GL_2(\bA_F)$ of weight $\tk$ and exact level $\fq$ with trivial central character. To each $\pi \in \cF(\tk,\fq)$ denote by $L(s,\pi)$ the standard $L$-function of $\pi$, and denote by $L(s,\pi,\Ad)$ the adjoint $L$-function of $\pi$. Let $L^{(\fq)}(s,\pi,\Ad)$ be the partial adjoint $L$-function with the local factor $L_\fq(s,\pi,\Ad)$ at $\fq$ being removed. Notice that, for $\pi\in\cF(\tk,\fq),$ $L_{\mathfrak{q}}(1,\pi,\Ad)=\zeta_{\fq}(2).$

\section{Moments of Central L-values}\label{section 2}
Let $\fn \subset \cO_F$ be an integral ideal coprime to $\fq$. For $\pi \in \cF(\tk,\fq)$, let $\lambda_\pi(\fn)$ be the $\fn^{th}$-Hecke eigenvalue of $\pi$. For $r \geq 1$, we define the $\lambda_\pi(\fn)$-weighted $r^{th}$-moment of central $L$-values $L(1/2,\pi)$ to be
\begin{align*}
    M_r(\tk,\fq,\fn) &:= \sum_{\pi \in \cF(\tk,\fq)} \frac{L(1/2,\pi)^r}{L^{(\fq)}(1,\pi,\Ad)}\lambda_\pi(\fn), \\
    \cM_r(\tk,\fq,\fn) &:= \sum_{\pi \in \cF(\tk,\fq)} L(1/2,\pi)^r \lambda_\pi(\fn).
\end{align*}
In this section, our goal is to establish asymptotic formulas for $M_r(\tk,\fq,\fn)$ and $\mathcal{M}_r(\tk,\fq,\fn)$ when $r=1,2,3.$ We will establish the formulas for $M_r(\tk,\fq,\fn)$ in \S \ref{subsec. first moment}, \ref{subsec. second moment} and \ref{subsec. cubic moment}. We will then remove the harmonic weight $L^{(\fq)}(1,\pi,\Ad)^{-1}$ to get formulas for $\cM_r(\tk,\fq,\fn)$ in \S \ref{remove the harmonic weight}, following the methods in \cite{KowalskiMichel1999}.

Later we will apply the formulas for $M_r(\mathbf{k},\fq,\fn)$ and $\cM_r(\mathbf{k},\fq,\fn)$ in the following two cases:
\begin{itemize}
	\item $\fn=\fp^{\ell}$ for some \textit{fixed} prime ideal $\fp$ and \textit{fixed} $\ell\geq0.$ This will be used to prove Theorem \ref{thm. sato tate, with harmonic weight} and Theorem \ref{thm. sato tate, without harmonic weight}
	\item $\fn=\fp$ for some prime ideal satisfying $N(\fp)\ll (\|\mathbf{k}\|^2N(\fq))^{\eta}$, with $\eta>0.$ This will be used to prove Theorem \ref{thm. llz, with harmonic weight} and Theorem \ref{thm. llz, without harmonic weight}.
\end{itemize}

\subsection{First Moment}\label{subsec. first moment}
In this subsection we study $M_1(\tk,\fq,\fn)$. We set
\[C_1(\mathbf{k},\fq):=2D_F^\frac{3}{2} \prod_{v|\infty} \frac{k_v-1}{4\pi^2}\zeta_\fq(2)^2(1-2\delta_\tk\delta_{\fq\subsetneq\cO_F}+N(\fq)). \]

Our first moment formula for $M_1(\tk,\fq,\fn)$ is as follows.

\begin{prop}\label{prop. first moment with harmonic weight}
Let notation be as before. Then we have
    \begin{equation*}
        M_1(\tk,\fq,\fn) =\frac{C_1(\mathbf{k},\fq)}{N(\fn)^\frac{1}{2}}+ O(\|\tk\|^\varepsilon N(\fq)^\varepsilon N(\fn)^{\frac{1}{2}+\varepsilon})
    \end{equation*}
    for any $\ve > 0$. Here the implied constant depends on $F$ and $\ve$ only.
\end{prop}

\begin{proof}
(Sketch) This is by applying \cite[Corollary 9.9]{WeiYangZhao2024} and \cite[Lemma 9.10]{WeiYangZhao2024} directly.
\end{proof}

\subsection{Second Moment}\label{subsec. second moment}
In this subsection we study $M_2(\tk,\fq,\fn)$. We set
\begin{align*}
c_2(\mathbf{k},\fq)&:=2D_F^\frac{3}{2}\gamma_{-1} \prod_{v|\infty} \frac{k_v-1}{4\pi^2}\zeta_{\fq}(2)^2, \\
C_2(\mathbf{k},\fq)&:=c_2(\tk,\fq) \cdot \left(N(\fq)+1-\frac{4\delta_\tk\delta_{\fq\subsetneq\cO_F}}{1+N(\fq)^{-1}} \right), \\
c_0(\tk)&:=2\sum_{v|\infty} \frac{\Gamma^\prime}{\Gamma} \left(\frac{k_v}{2}\right) + \frac{2\gamma_0}{\gamma_{-1}} + 2\log \left(\frac{D_F}{(2\pi)^{d_F}}\right).
\end{align*}

Recall the following function $G_{\tk,\fq,\fn}(s)$ that is used in \cite{WeiYangZhao2024}:
\begin{equation*}
    G_{\tk,\fq,\fn}(s) := 2D_F^\frac{3}{2} \frac{\tau(\fn)}{N(\fn)^\frac{1}{2}} \prod_{v|\infty} \frac{k_v-1}{4\pi^2}(1+N(\fq)^s)\left(\frac{D_F}{(2\pi)^{d_F}N(\fn)^\frac{1}{2}}\right)^s \prod_{v|\infty} \frac{\Gamma((s+k_v)/2)^2}{\Gamma(k_v/2)^2}.
\end{equation*}

Our second moment formula for $M_2(\tk,\fq,\fn)$ is as follows.
\begin{prop}\label{second moment}
Let notation be as before. Then we have
\begin{align*}
        M_2(\tk,\fq,\fn) &=c_2(\mathbf{k},\fq) \frac{\tau(\fn)}{N(\fn)^\frac{1}{2}}(N(\fq)+1)\left(c_0(\tk) + \log N(\fq)-\log N(\fn)\right) \\
        &- 4c_2(\mathbf{k},\fq)\frac{\tau(\fn)}{N(\fn)^\frac{1}{2}} \frac{\delta_\tk\delta_{\fq\subsetneq\cO_F}}{1+N(\fq)^{-1}} \left(c_0(\tk)+\frac{2\log N(\fq)}{N(\fq)+1}-\log N(\fn)\right) \\
        &+ O(\|\tk\|^{\frac{1}{2}+\varepsilon} N(\fq)^\varepsilon N(\fn)^{\frac{1}{2}+\varepsilon})
    \end{align*}
    for any $\ve > 0$. Here the implied constant depends on $\varepsilon$ and $F$ only.
\end{prop}

\begin{proof}
	(Sketch) Combining \cite[Corollary 8.2]{WeiYangZhao2024} and \cite[Lemma 8.3]{WeiYangZhao2024}, we have
\begin{flalign*}
&\frac{1}{\zeta_{\fq}(2)^2}M_2(\mathbf{k},\fq,\fn)= O(\|\tk\|^{\frac{1}{2}+\varepsilon} N(\fq)^\varepsilon N(\fn)^{\frac{1}{2}+\varepsilon})\\
&\hspace{10mm}+(N(\fq)+1) \cdot \Res_{s=0}\left(\frac{\zeta_F(1+s)G_{\tk,\fq,\fn}(s)}{s}\right)-4\frac{N(\fq)+1}{N(\fq)}\delta_\tk\delta_{\fq\subsetneq\cO_F} \cdot \Res_{s=0} \left(\frac{\zeta_F(1+s)G_{\tk,\cO_F,\fn}(s)}{s(1+N(\fq)^{-1-s/2})^2}\right).
\end{flalign*}
	Then one can compute the residues directly to get Proposition \ref{second moment}.
\end{proof}

We state the following simplified form of Proposition \ref{second moment} for future reference. Recall that we have set $Q_2 = \|\tk\|^2N(\fq)$.
\begin{cor}\label{cor. second moment with harmonic weight}
Let notation be as before. Then we have 
\[M_{2}(\mathbf{k},\fq,\fn)=C_2(\mathbf{k},\fq)\frac{\tau(\fn)}{N(\fn)^\frac{1}{2}}(\log Q_2-\log N(\fn)+O(1))+O(\|\tk\|^{\frac{1}{2}+\varepsilon} N(\fq)^\varepsilon N(\fn)^{\frac{1}{2}+\varepsilon})\]
for any $\ve > 0$. Here the implied constant depends on $\varepsilon$ and $F$ only.
\end{cor}

\subsection{Cubic Moment}\label{subsec. cubic moment}
In this subsection we fix $\tk$ and study $M_3(\tk,\fq,\fn)$ as $N(\fq)$ being large. We start with the following functional equation of $L(s,\pi)$:
\begin{align*}
    \Lambda(s,\pi) := (D_F^2 N(\fq))^\frac{s}{2}(2\pi)^{-d_F s}\prod_{v|\infty} \Gamma\left(s+\frac{k_v-1}{2}\right)L(s,\pi) = \varepsilon_\pi \Lambda(1-s,\pi).
\end{align*}
Here $\varepsilon_\pi = \pm 1$ is the root number. A standard approximate functional equation argument gives
\begin{equation}\label{approximate functional equation}
    L(1/2,\pi) = (1+\varepsilon_\pi)\sum_{\fm \subset \cO_F} \frac{\lambda_\pi(\fm)}{N(\fm)^\frac{1}{2}} V\left(\frac{N(\fm)}{N(\fq)^\frac{1}{2}}\right),
\end{equation}
where the function $V$ is defined by 
\begin{equation*}
    V(y) := \frac{1}{2\pi i}\int_{(2)} \prod_{v|\infty} \frac{\Gamma(s+k_v/2)}{\Gamma(k_v/2)}\left(\frac{(2\pi)^{d_F}y}{D_F}\right)^{-s} G(s) \frac{ds}{s}
\end{equation*}
for a choice of an even holomorphic function $G(s)$ that is bounded in $-3 \leq \sigma \leq 3$ with $G(0) = 1$. One can choose such a $G(s)$ so that 
\begin{equation*}
    V(y) \ll_A \left(1+\frac{y}{\|\tk\|}\right)^{-A}, \hspace{3mm} y > 0
\end{equation*}
for any $A > 0$. By this bound we may assume the sum in (\ref{approximate functional equation}) is taken over $\fm$ such that $N(\fm) \ll N(\fq)^{\frac{1}{2}+\varepsilon}$, up to an error of $N(\fq)^{-100}$. 

By (\ref{approximate functional equation}) we write the cubic moment $M_3(\tk,\fq,\fn)$ as 
\begin{equation*}
    M_3(\tk,\fq,\fn) = 2\sum_{N(\fm) \ll N(\fq)^{\frac{1}{2}+\varepsilon}} \frac{1}{N(\fm)^\frac{1}{2}} V\left(\frac{N(\fm)}{N(\fq)^\frac{1}{2}}\right)\sum_{\pi \in \cF(\tk,\fq)}\frac{L(1/2,\pi)^2}{L^{(\fq)}(1,\pi,\Ad)}\lambda_\pi(\fn)\lambda_\pi(\fm) + O(N(\fq)^{-10}N(\fn)^{\frac{1}{2}+\varepsilon}),
\end{equation*}
where (and hence after in this subsection) the implied constants may depend on $\tk$, $\varepsilon$ and $F$.

Note that $N(\fm) \ll N(\fq)^{\frac{1}{2}+\varepsilon}$ implies $(\fm,\fq)=1$ for $N(\fq)$ large enough, since $\fq$ is prime. By the Hecke relation
\begin{equation*}
\lambda_\pi(\fn)\lambda_\pi(\fm) = \sum_{\fc|(\fn,\fm)} \lambda_\pi(\fn\fm\fc^{-2})
\end{equation*}
we can write 
\begin{equation*}
    M_3(\tk,\fq,\fn) = 2\sum_{N(\fm) \ll N(\fq)^{\frac{1}{2}+\varepsilon}} \frac{1}{N(\fm)^\frac{1}{2}} V\left(\frac{N(\fm)}{N(\fq)^\frac{1}{2}}\right) \sum_{\fc|(\fn,\fm)} M_2(\tk,\fq,\fn\fm\fc^{-2}) + O(N(\fq)^{-10}N(\fn)^{\frac{1}{2}+\varepsilon}).
\end{equation*}
Applying Proposition \ref{second moment} to $M_2(\tk,\fq,\fn\fm\fc^{-2})$, we get
\begin{equation*}
    M_2(\tk,\fq,\fn\fm\fc^{-2}) = C_2(\mathbf{k},\fq) \frac{\tau(\fn\fm\fc^{-2})}{N(\fn\fm\fc^{-2})^\frac{1}{2}}(\log N(\fq)-\log N(\fn\fm\fc^{-2})+c_0(k)) + O(N(\fq)^\varepsilon N(\fn\fm\fc^{-2})^{\frac{1}{2}+\varepsilon}).
\end{equation*}
Thus we have
\begin{align*}
    M_3(\tk,\fq,\fn) &= 2C_2(\mathbf{k},\fq)(\log N(\fq)+c_0(k))\sum_{(\fm,\fq)=1}\frac{1}{N(\fm)^\frac{1}{2}}V\left(\frac{N(\fm)}{N(\fq)^\frac{1}{2}}\right) \sum_{\fc|(\fn,\fm)}\frac{\tau(\fn\fm\fc^{-2})}{N(\fn\fm\fc^{-2})^\frac{1}{2}} \\
    &- 2C_2(\mathbf{k},\fq)\sum_{(\fm,\fq)=1}\frac{1}{N(\fm)^\frac{1}{2}}V\left(\frac{N(\fm)}{N(\fq)^\frac{1}{2}}\right) \sum_{\fc|(\fn,\fm)}\frac{\tau(\fn\fm\fc^{-2})}{N(\fn\fm\fc^{-2})^\frac{1}{2}} \log N(\fn\fm\fc^{-2}) \\
    &+ O(N(\fq)^{\frac{1}{2}+\varepsilon} N(\fn)^{\frac{1}{2}+\varepsilon}).
\end{align*}

We now set 
\begin{align*}
    S_1(\fq,\fn) &:= \sum_{(\fm,\fq)=1}\frac{1}{N(\fm)^\frac{1}{2}}V\left(\frac{N(\fm)}{N(\fq)^\frac{1}{2}}\right) \sum_{\fc|(\fn,\fm)}\frac{\tau(\fn\fm\fc^{-2})}{N(\fn\fm\fc^{-2})^\frac{1}{2}}, \\
    S_2(\fq,\fn) &:= \sum_{(\fm,\fq)=1}\frac{1}{N(\fm)^\frac{1}{2}}V\left(\frac{N(\fm)}{N(\fq)^\frac{1}{2}}\right) \sum_{\fc|(\fn,\fm)}\frac{\tau(\fn\fm\fc^{-2})}{N(\fn\fm\fc^{-2})^\frac{1}{2}} \log N(\fn\fm\fc^{-2})
\end{align*}
and evaluate them. For $S_1(\fq,\fn)$ we change the summation order and change variable $\fm \to \fc\fm$ to get
\begin{equation*}
    S_1(\fq,\fn) = \frac{1}{N(\fn)^\frac{1}{2}} \sum_{\fc|\fn}\sum_{(\fm,\fq)=1}\frac{1}{N(\fm)}V\left(\frac{N(\fm)N(\fc)}{N(\fq)^\frac{1}{2}}\right)\tau(\fn\fc^{-1}\fm).
\end{equation*}
We cite the following property of the divisor function:
\begin{equation*}
    \tau(\fm\fn) = \sum_{\fd|(\fm,\fn)}\mu(\fd)\tau(\fm\fd^{-1})\tau(\fn\fd^{-1}).
\end{equation*}
Applying this and changing variable $\fm \to \fd\fm$, we have
\begin{equation*}
    S_1(\fq,\fn) = \frac{1}{N(\fn)^\frac{1}{2}} \sum_{\fc|\fn} \sum_{\fd|\fn\fc^{-1}} \frac{\mu(\fd)\tau(\fn(\fc\fd)^{-1})}{N(\fd)}\sum_{(\fm,\fq)=1}\frac{\tau(\fm)}{N(\fm)}V\left(\frac{N(\fm)N(\fc\fd)}{N(\fq)^\frac{1}{2}}\right).
\end{equation*}
The inner summation over $\fm$ is
\begin{align*}
    \sum_{(\fm,\fq)=1}\frac{\tau(\fm)}{N(\fm)}V\left(\frac{N(\fm)N(\fc\fd)}{N(\fq)^\frac{1}{2}}\right) 
    &= \frac{1}{2\pi i}\int_{(2)}\sum_{(\fm,\fq)=1}\frac{\tau(\fm)}{N(\fm)^{1+s}} \prod_{v|\infty} \frac{\Gamma(s+k_v/2)}{\Gamma(k_v/2)}\left(\frac{(2\pi)^{d_F}N(\fc\fd)}{D_F N(\fq)^\frac{1}{2}}\right)^{-s}G(s)\frac{ds}{s} \\
    &= \frac{1}{2\pi i}\int_{(2)}\zeta_F^{(\fq)}(1+s)^2 \prod_{v|\infty} \frac{\Gamma(s+k_v/2)}{\Gamma(k_v/2)}\left(\frac{(2\pi)^{d_F}N(\fc\fd)}{D_F N(\fq)^\frac{1}{2}}\right)^{-s}G(s)\frac{ds}{s}.
\end{align*}
The integrand has a triple pole at $s=0$. We shift integral contour to $\sigma = -1$. It is directly computed that the residue at $s=0$ is

\begin{equation*}
    \frac{\gamma_{-1}^2}{2}\left(\frac{1}{2}\log N(\fq)-\log N(\fc\fd)\right)^2+ O(\log N(\fq) N(\fc\fd))
\end{equation*}
and the shifted integral to $\sigma = -1$ is 
\begin{equation*}
    O(N(\fq)^{-\frac{1}{2}}N(\fc\fd)).
\end{equation*}
Thus we have the following formula for $S_1(\fq,\fn)$:
\begin{align*}
    S_1(\fq,\fn) &= \frac{\gamma_{-1}^2}{2} \cdot N(\fn)^{-\frac{1}{2}}\sum_{\fc|\fn} \sum_{\fd|\fn\fc^{-1}} \frac{\mu(\fd)\tau(\fn(\fc\fd)^{-1})}{N(\fd)}\left(\frac{\log N(\fq)}{2}-\log N(\fc\fd)\right)^2 \\
    &\hspace{10mm}+ O\left(N(\fn)^{-\frac{1}{2}}\sum_{\fc|\fn} \sum_{\fd|\fn\fc^{-1}} \frac{|\mu(\fd)|\tau(\fn(\fc\fd)^{-1})}{N(\fd)}\log (N(\fc\fd)N(\fq))\right) \\
    &\hspace{10mm}+ O\left(N(\fq)^{-\frac{1}{2}} \cdot N(\fn)^{-\frac{1}{2}} \sum_{\fc|\fn} N(\fc)\sum_{\fd|\fn\fc^{-1}} |\mu(\fd)|\tau(\fn(\fc\fd)^{-1})\right)\\
    &=\frac{\gamma_{-1}^2}{2} \cdot N(\fn)^{-\frac{1}{2}}\sum_{\fc|\fn}\tau(\fn\fc^{-1}) \left(\frac{\log N(\fq)}{2}-\log N(\fc)\right)^2\sum_{\fd|\fc} \frac{\mu(\fd)}{N(\fd)}\\
    &\hspace{10mm}+O\left(\frac{\tau(\fn)^2\log N(\fq)}{N(\fn)^{1/2}}\right)+O(N(\fq)^{-1/2+\varepsilon}N(\fn)^{1/2+\varepsilon})
\end{align*}
Here we used the following bound
\[\sum_{\fc|\fn} \sum_{\fd|\fn\fc^{-1}} \frac{|\mu(\fd)|\tau(\fn(\fc\fd)^{-1})}{N(\fd)}\leq \tau(\fn)^2.\]
Notice that the left hand side is multiplicative in $\fn$. Thus it suffices to investigate the case when $\fn$ is a prime power.

We now analyze the term $S_2(\fq,\fn)$. By similar arguments we have
\begin{align*}
    S_2(\fq,\fn) &= \frac{1}{N(\fn)^\frac{1}{2}} \sum_{\fc|\fn} \sum_{\fd|\fn\fc^{-1}} \frac{\mu(\fd)\tau(\fn(\fc\fd)^{-1})}{N(\fd)} \log N(\fn\fc^{-1}\fd) \sum_{(\fm,\fq)=1}\frac{\tau(\fm)}{N(\fm)}V\left(\frac{N(\fm)N(\fc\fd)}{N(\fq)^\frac{1}{2}}\right) \\
    &+ \frac{1}{N(\fn)^\frac{1}{2}} \sum_{\fc|\fn} \sum_{\fd|\fn\fc^{-1}} \frac{\mu(\fd)\tau(\fn(\fc\fd)^{-1})}{N(\fd)}\sum_{(\fm,\fq)=1}\frac{\tau(\fm)\log N(\fm)}{N(\fm)}V\left(\frac{N(\fm)N(\fc\fd)}{N(\fq)^\frac{1}{2}}\right).
\end{align*}
The sum over $\fm$ in the first line was analyzed above. We now analyze the second sum over $\fm$:
\begin{align*}
    &\sum_{(\fm,\fq)=1}\frac{\tau(\fm)\log N(\fm)}{N(\fm)}V\left(\frac{N(\fm)N(\fc\fd)}{N(\fq)^\frac{1}{2}}\right) \\
    &\hspace{20mm}= \frac{1}{2\pi i}\int_{(2)}\sum_{(\fm,\fq)=1}\frac{\tau(\fm)\log N(\fm)}{N(\fm)^{1+s}} \prod_{v|\infty} \frac{\Gamma(s+k_v/2)}{\Gamma(k_v/2)}\left(\frac{(2\pi)^{d_F}N(\fc\fd)}{D_F N(\fq)^\frac{1}{2}}\right)^{-s}G(s)\frac{ds}{s} \\
    &\hspace{20mm}= \frac{-1}{2\pi i}\int_{(2)}(\zeta_F^{(\fq)}(1+s)^2)^\prime \prod_{v|\infty} \frac{\Gamma(s+k_v/2)}{\Gamma(k_v/2)}\left(\frac{(2\pi)^{d_F}N(\fc\fd)}{D_F N(\fq)^\frac{1}{2}}\right)^{-s}G(s)\frac{ds}{s}.
\end{align*}
We again move the integration line to $\sigma  = -1$ and pick up residue at $s=0$. Note that here the integrand has a pole of order 4 at $s=0$. The residue at $s=0$ is
\begin{equation*}
    -\frac{\gamma_{-1}^2}{3}\left(\frac{\log N(\fq)}{2}-\log N(\fc\fd)\right)^3 + O(\log^2(N(\fq)N(\fc\fd))).
\end{equation*}
We thus have the following formula for $S_2(\fq,\fn)$:
\begin{flalign*}
    S_2(\fq,\fn) &= \frac{\gamma_{-1}^2}{3} \cdot N(\fn)^{-\frac{1}{2}}\sum_{\fc|\fn} \sum_{\fd|\fn\fc^{-1}} \frac{\mu(\fd)\tau(\fn(\fc\fd)^{-1})}{N(\fd)}\left(\frac{\log N(\fq)}{2}-\log N(\fc\fd)\right)^3 \\
    &\hspace{10mm}+ \frac{\gamma_{-1}^2}{2} \cdot N(\fn)^{-\frac{1}{2}}\sum_{\fc|\fn} \sum_{\fd|\fn\fc^{-1}} \frac{\mu(\fd)\tau(\fn(\fc\fd)^{-1})}{N(\fd)} \log N(\fn\fc^{-1}\fd) \left(\frac{\log N(\fq)}{2}-\log N(\fc\fd)\right)^2 \\
    &\hspace{10mm}+ O\left(N(\fn)^{-\frac{1}{2}}\sum_{\fc|\fn} \sum_{\fd|\fn\fc^{-1}} \frac{|\mu(\fd)|\tau(\fn(\fc\fd)^{-1})}{N(\fd)}\log^2 (N(\fc\fd)N(\fq))\right) \\
    &\hspace{10mm}+ O\left(N(\fn)^{-\frac{1}{2}}\sum_{\fc|\fn} \sum_{\fd|\fn\fc^{-1}} \frac{|\mu(\fd)|\tau(\fn(\fc\fd)^{-1})}{N(\fd)} \log N(\fn\fc^{-1}\fd) \log (N(\fc\fd)N(\fq))\right) \\
    &\hspace{10mm}+ O\left(N(\fq)^{-\frac{1}{2}} \cdot N(\fn)^{-\frac{1}{2}} \sum_{\fc|\fn} N(\fc)\sum_{\fd|\fn\fc^{-1}} |\mu(\fd)|\tau(\fn(\fc\fd)^{-1})\log N(\fn\fc^{-1}\fd)\right)\\
    &= \frac{\gamma_{-1}^2}{3} \cdot N(\fn)^{-\frac{1}{2}}\sum_{\fc|\fn}\tau(\fn\fc^{-1})\left(\frac{\log N(\fq)}{2}-\log N(\fc)\right)^3 \sum_{\fd|\fc} \frac{\mu(\fd)}{N(\fd)} \\
    &\hspace{10mm}+ \frac{\gamma_{-1}^2}{2} \cdot N(\fn)^{-\frac{1}{2}}\sum_{\fc|\fn}  \tau(\fn\fc^{-1})\log N(\fn\fc^{-1}) \left(\frac{\log N(\fq)}{2}-\log N(\fc)\right)^2\sum_{\fd|\fc} \frac{\mu(\fd)}{N(\fd)} \\
    &\hspace{10mm}+O\left(\frac{\log^2 N(\fq)\tau(\fn)^2}{N(\fn)^{1/2}}\right)+O(N(\fq)^{-1/2+\varepsilon}N(\fn)^{1/2+\varepsilon}).
\end{flalign*}

We summarize the above calculation as follows:
\begin{prop}\label{cubic moment}
Let notation be as before. Let $\mathbf{k}$ be fixed. Then for any $\varepsilon>0$ we have
\begin{equation*}
    M_3(\tk,\fq,\fn) = 2C_2(\mathbf{k},\fq)\cdot\{(\log N(\fq)+c_0(k))S_1(\fq,\fn) - S_2(\fq,\fn)\} + O(N(\fq)^{\frac{1}{2}+\varepsilon}N(\fn)^{\frac{1}{2}+\varepsilon}).
\end{equation*}
Here the implied constant depends on $\tk,\varepsilon$ and $F$. The definitions and asymptotic formulas for $S_1(\fq,\fn)$ and $S_2(\fq,\fn)$ can be found above. 
\end{prop}

This formula looks complicated. But eventually we will focus on special cases where either $\fn = \fp^\ell$ is a fixed prime power, or $\fn = \fp$ is a prime with bounded norm. In both cases the asymptotic formula for $M_3(\tk,\fq,\fn)$ will take a much nicer shape.

\subsection{Removing the Harmonic Weight}\label{remove the harmonic weight}
In this subsection, we follow the methods in \cite{KowalskiMichel1999} to remove the harmonic weight $L^{(\fq)}(1,\pi,\Ad)^{-1}$ in $M_r(\tk,\fq,\fn)$ to obtain formulas for $\cM_r(\tk,\fq,\fn)$ when $r=1,2$. 

For $\pi\in \cF(\mathbf{k},\fq)$, by the approximate functional equation, we have the decomposition
\[L^{(\fq)}(1,\pi,\Ad)=\omega_{\pi}(x)+\omega_{\pi}(x,y)+O((\|\mathbf{k}\|N(\fq))^{-2+\varepsilon})\]
for $1\leq x<y$. Here $\omega_\pi(x,y)$ is defined by
\[\omega_{\pi}(x,y)=\sum_{\substack{x<N(\fl\ff^2)\leq y\\(\fl\ff,\fq)=1}}\frac{\lambda_{\pi}(\fl^2)}{N(\fl\ff^2)}\]
and $\omega_{\pi}(x)=\omega_{\pi}(0,x),$ provided $y\geq (\|\mathbf{k}\|N(\fq)))^{10}.$

Let $\{\alpha_{\pi}\}_{\pi \in \cF(\tk,\fq)}$ be a family of complex numbers. Then we can write
\[\sum_{\pi\in\cF(\mathbf{k},\fq)}\alpha_{\pi}=\sum_{\pi\in\cF(\mathbf{k},\fq)}\frac{\alpha_{\pi}L^{(\fq)}(1,\pi,\Ad)}{L^{(\fq)}(1,\pi,\Ad)}=\sum_{\pi\in\cF(\mathbf{k},\fq)}\frac{\omega_{\pi}(x)\alpha_{\pi}}{L^{(\fq)}(1,\pi,\Ad)}+\sum_{\pi\in\cF(\mathbf{k},\fq)}\frac{\omega_{\pi}(x,y)\alpha_{\pi}}{L^{(\fq)}(1,\pi,\Ad)}+O(1).\]
By the methods in \cite[Section~3]{KowalskiMichel1999}, if we can show
\begin{align}\label{conditions to estimate the middle sum}
\begin{split}
\sum_{\pi\in\cF(\mathbf{k},\fq)}\frac{|\alpha_{\pi}|}{L^{(\fq)}(1,\pi,\Ad)}&\ll (\|\mathbf{k}\|N(\fq))\log^A(\|\mathbf{k}\|N(\fq)) \\
\max_{\pi\in\cF(\mathbf{k},\fq)}\{\alpha_{\pi}\}&\ll (\|\mathbf{k}\|N(\fq))^{1-\eta}  \\
\end{split}
\end{align}
for arbitrary $A>0$ and some $\eta>0,$ then we have
\begin{equation}\label{the estimate for the middle sum in adjoint}
\sum_{\pi\in\cF(\mathbf{k},\fq)}\frac{\omega_{\pi}(x,y)\alpha_{\pi}}{L^{(\fq)}(1,\pi,\Ad)}\ll (\|\mathbf{k}\|N(\fq))^{1-\eta'}.
\end{equation}
for some $\eta' > 0$ that depends on $\eta$ and $x.$ In particular, if we take $x=(\|\mathbf{k}\|N(\fq))^{\varepsilon}$ for a sufficiently small $\ve > 0$, then $\eta'$ is dependent on $\eta$ and $\varepsilon$. Therefore, we have 
\[\sum_{\pi\in\cF(\mathbf{k},\fq)}\alpha_{\pi}=\sum_{\pi\in\cF(\mathbf{k},\fq)}\frac{\omega_{\pi}(x)\alpha_{\pi}}{L^{(\fq)}(1,\pi,\Ad)}+O((\|\mathbf{k}\|N(\fq))^{1-\eta'}).\]

Apply the above method of removing the harmonic weight to $\alpha_\pi = L(1/2,\pi)\lambda_\pi(\fn)$, and use Proposition \ref{prop. first moment with harmonic weight}, we can prove the following formula for $\cM_1(\mathbf{k},\fq,\fn).$
\begin{prop}\label{prop. first moment without harmonic weight}
Let notation be as before. Let $\fn$ be an integral ideal coprime to $\fq$. Then for any $\ve > 0$ we have 
    \begin{equation*}
        \cM_1(\tk,\fq,\fn) =\frac{C_1(\mathbf{k},\fq)}{N(\fn)^{1/2}}\sum_{\fr|\fn}\sum_{\substack{N(\fl\ff^2)\ll Q_1^{\varepsilon} \\ (\fl\ff,\fq)=1, \fr|\fl^2}}\frac{N(\fr)}{N(\fl)^2N(\ff)^2}+O((\|\mathbf{k}\|N(\fq))^{1-\eta'})+ O(\|\tk\|^\varepsilon N(\fq)^\varepsilon N(\fn)^{\frac{1}{2}+\varepsilon}),
    \end{equation*}
    for some $\eta^\prime > 0$ depending on $\epsilon$, and the implied constants depend on $\epsilon$ and $F$.
\end{prop}
\begin{proof}
By the Hecke relation, we have $\lambda_{\pi}(\fn)\lambda_{\pi}(\fl^2)=\sum_{\fr|(\fn,\fl^2)}\lambda_{\pi}(\fn\fl^2\fr^{-2}).$ By Proposition \ref{prop. first moment with harmonic weight} and the convexity bound for $L(1/2,\pi),$ $\alpha_{\pi}=L(1/2,\pi)\lambda_\pi(\fn)$ satisfies \eqref{conditions to estimate the middle sum}. Then we have:
	\[\cM_1(\tk,\fq,\fn)=\sum_{\substack{N(\fl\ff^2)\ll Q_1^{\varepsilon} \\ (\fl\ff,\fq)=1}}\frac{1}{N(\fl\ff^2)}\sum_{\fr|(\fn,\fl^2)}M_1(\mathbf{k},\fq,\fn\fl^2\fr^{-2})+O((\|\mathbf{k}\|N(\fq))^{1-\eta'}).\]
Apply Proposition \ref{prop. first moment with harmonic weight} to $M_1(\tk,\fq,\fn\fl^2\fr^{-2})$, we have
\[\cM_1(\tk,\fq,\fn)=C_1(\mathbf{k},\fq)\sum_{\substack{N(\fl\ff^2)\ll Q_1^{\varepsilon} \\ (\fl\ff,\fq)=1}}\frac{1}{N(\fl\ff^2)}\sum_{\fr|(\fn,\fl^2)}\frac{N(\fr)}{N(\fn\fl^2)^{1/2}}+O((\|\mathbf{k}\|N(\fq))^{1-\eta'})+ O(\|\tk\|^\varepsilon N(\fq)^\varepsilon N(\fn)^{\frac{1}{2}+\varepsilon}).\]
Change the order of summation, and we complete the proof.
\end{proof}

Similarly, apply the removing harmonic weight method to $\alpha_\pi = L(1/2,\pi)^2\lambda_\pi(\fn)$, and use Corollary \ref{cor. second moment with harmonic weight}, we prove the following result for $\cM_2(\mathbf{k},\fq,\fn)$

\begin{prop}\label{prop. second moment without harmonic weight}
Let notation be as before. Let $\fn$ be an integral ideal coprime to $\fq$. Then for any $\ve > 0$ we have
  \begin{flalign*}
  	\cM_2(\tk,\fq,\fn)&=\frac{C_2(\mathbf{k},\fq)}{N(\fn)^{1/2}}\sum_{\fr|\fn}\tau(\fn\fr^{-1})\sum_{\substack{N(\fl\ff^2)\ll Q_2^{\varepsilon} \\ (\fl\ff,\fq)=1,\fr|\fl^2}}\frac{\tau(\fl^2\fr^{-1})}{N(\fl)^2N(\ff)^2}\sum_{\fd|\fr}(\log Q_2-\log N(\fn\fl^2\fd^{-2}))N(\fd)\mu(\fr\fd^{-1})\\
  	&\hspace{20mm}+O\left(\frac{\tau(\fn)\|\mathbf{k}\|N(\fq)}{N(\fn)^{1/2}}\right)+O((\|\mathbf{k}\|N(\fq))^{1-\eta'})
  \end{flalign*}
  for some $\eta^\prime > 0$ depending on $\epsilon$, and the implied constants depend on $\epsilon$ and $F$.
 \end{prop}
\begin{proof}
	By Proposition \ref{second moment} and \cite[Theorem 11.3]{WeiYangZhao2024}, $\alpha_{\pi}=L(1/2,\pi)^2\lambda_\pi(\fn)$ satisfies \eqref{conditions to estimate the middle sum}. Then we have:
\begin{flalign*}
\cM_2(\tk,\fq,\fn)&=\sum_{\substack{N(\fl\ff^2)\ll Q_2^{\varepsilon} \\ (\fl\ff,\fq)=1}}\frac{1}{N(\fl\ff^2)}\sum_{\fr|(\fn,\fl^2)}M_2(\mathbf{k},\fq,\fn\fl^2\fr^{-2})+O((\|\mathbf{k}\|N(\fq))^{1-\eta'})\\
&=C_2(\mathbf{k},\fq)\sum_{\substack{N(\fl\ff^2)\ll Q_2^{\varepsilon} \\ (\fl\ff,\fq)=1}}\frac{1}{N(\fl\ff^2)}\sum_{\fr|(\fn,\fl^2)}\frac{\tau(\fn\fl^2\fr^{-2})N(\fr)}{N(\fn)^{1/2}N(\fl)}(\log Q_2-\log N(\fn\fl^2\fr^{-2})+O(1))\\
&\hspace{20mm}+O((\|\mathbf{k}\|N(\fq))^{1-\eta'})
\end{flalign*}
Apply $\tau(\fm\fn)=\sum_{\fd|(\fm,\fn)}\mu(\fd)\tau(\fn\fd^{-1})\tau(\fm\fd^{-1})$ and reorder the summation:
\[\sum_{\fr|(\fn,\fl^2)}\frac{\tau(\fn\fl^2\fr^{-2})N(\fr)}{N(\fn)^{1/2}N(\fl)}(\log Q_2-\log N(\fn\fl^2\fr^{-2}))=\sum_{\fr|(\fn,\fl^2)}\frac{\tau(\fn\fr^{-1})\tau(\fl^2\fr^{-1})}{N(\fn)^{1/2}N(\fl)}\sum_{\fd|\fr}(\log Q_2-\log N(\fn\fl^2\fd^{-2}))N(\fd)\mu(\fr\fd^{-1}).\]
Insert it into $\cM_2(\mathbf{k},\fq,\fn)$ and change the summation order, we complete the proof.
\end{proof}

\section{Weighted Equidistribution of $\lambda_\pi(\fp)$}\label{sec, weighted sato tate}
Let $\fp$ be a prime ideal distinct from $\fq$. Our goal of this section is to find measures $\mu_{\fp,r}$ and $\mu_{\fp,r}^{un}$ on $\bR$ ($r=1,2,3$) such that 
\begin{align*}
\lim_{Q_r \to \infty} \frac{1}{\sum_{\pi \in \cF(\tk,\fq)}\frac{L(1/2,\pi)^r}{L(1,\pi,\Ad)}}\sum_{\pi \in \cF(\tk,\fq)}\frac{L(1/2,\pi)^r}{L(1,\pi,\Ad)}\phi(\lambda_\pi(\fp)) &= \int_\bR\phi(x)\,d\mu_{\fp,r}(x), \\
\lim_{Q_r \to \infty} \frac{1}{\sum_{\pi \in \cF(\tk,\fq)}L(1/2,\pi)^r}\sum_{\pi \in \cF(\tk,\fq)}L(1/2,\pi)^r\phi(\lambda_\pi(\fp)) &= \int_\bR\phi(x)\,d\mu_{\fp,r}^{un}(x),
\end{align*}
for any continuous function $\phi$ on $\bR$. We will only prove the case $r=2$ and $r=3$ since the proof of the $r=1$ case is much easier.

Note that we have the bound $|\lambda_\pi(\fp)| \leq 2$ \cite{Blasius2006}. Thus by Weierstrass Approximation Theorem, it suffices to check for polynomial functions $\phi$. Further, by linearity it suffices to check for a family of polynomials that spans $\bR[x]$ as a $\bR$-vector space. We choose the family of Chebyshev polynomials.

\subsection{Chebyshev Polynomials}\label{subsec, cheby poly}
The family of Chebyshev polynomials $\{T_n(x)\}_{n \geq 0}$ is defined by the following recurrence relations:
\begin{align*}
    T_0(x) &= 1, \\
    T_1(x) &= x, \\
    T_{n+1}(x) &= xT_n(x) - T_{n-1}(x), \hspace{3mm} n \geq 1.
\end{align*}
The $n^{th}$-Chebyshev polynomial $T_n(x)$ is of degree $n$. 
\begin{remark}\label{rem. usual Che Poly}
	Let $U_n(x)$ be the usual Chebyshev polynomial defined by $U_n(\cos\theta)=\frac{\sin(n+1)\theta}{\sin\theta}.$ Then $T_n(x)=U_n(x/2).$
\end{remark}
By comparing the above recurrence relation with the Hecke relation
\begin{equation*}
    \lambda_\pi(\fp^{n+1}) = \lambda_\pi(\fp)\lambda_\pi(\fp^n) - \lambda_{\pi}(\fp^{n-1})
\end{equation*}
for $\pi \in \cF(\tk,\fq)$, we see that $T_n(\lambda_\pi(\fp)) = \lambda_\pi(\fp^n)$.

Another useful property of Chebyshev polynomials is that they form an orthonormal family with respect to the Sato-Tate measure $\mu_\infty$:
\begin{equation*}
    \int_\bR T_n(x)T_m(x) d\mu_\infty(x) = \delta_{m=n}.
\end{equation*}
This can be proved by a change of variable $x=2\cos\theta$ and Remark \ref{rem. usual Che Poly}. This property can be used to prove the following lemma:

\begin{lemma}\label{lem, main lem for sato-tate with harmonic weight}
	Let $r\ge1$ be a positive integer. Suppose that for all $\ell\geq0,$ we have
	 \[\lim_{Q_r \to \infty} \frac{M_r(\tk,\fq,\fp^\ell)}{M_r(\tk,\fq,\cO_F)} = F_{\fp,r,\ell}.\]
	Then the desired measure $\mu_{\fp,r}$ in Theorem \ref{thm. sato tate, with harmonic weight} can be given by
	\[ d\mu_{\fp,r}(x) = \left(\sum_{\ell=0}^\infty F_{\fp,r,\ell} \cdot T_{\ell}(x)\right) d\mu_\infty(x).\]
\end{lemma}

\begin{remark}\label{rem, main lem for sato-tate without harmonic weight}
	This Lemma is also valid when $M_r$ is replaced by $\cM_r$, and $\mu_{\fp,r}$ is replaced by $\mu_{\fp,r}^{un}$. 
\end{remark}

\begin{proof}
We know that, $T_{\ell}(\lambda_{\pi}(\fp))=\lambda_{\pi}(\fp^{\ell}).$ Taking $\phi = T_\ell$ to be Chebyshev polynomials, then
\begin{align*}
\lim_{Q_r \to \infty} \frac{1}{\sum_{\pi \in \cF(\tk,\fq)}\frac{L(1/2,\pi)^r}{L(1,\pi,\Ad)}}\sum_{\pi \in \cF(\tk,\fq)}\frac{L(1/2,\pi)^r}{L(1,\pi,\Ad)}T_\ell(\lambda_\pi(\fp)) = \lim_{Q_r \to \infty} \frac{M_r(\tk,\fq,\fp^\ell)}{M_r(\tk,\fq,\cO_F)} = \int_\bR T_\ell(x)d\mu_{\fp,r}(x), \end{align*}
for any fixed prime ideal $\fp \neq \fq$ and $\ell \geq 0$. 

If we expand these measures $\mu_{\fp,r}$ using $\{T_{\ell}(x)\}_{\ell \geq 0}$ as a basis, we obtain:
\begin{equation}\label{Chebyshev polynomial series expansion}
    d\mu_{\fp,r}(x) = \left(\sum_{\ell=0}^\infty F_{\fp,r,\ell} \cdot T_{\ell}(x)\right) d\mu_\infty(x).
\end{equation}
This completes the proof.
\end{proof}

Once we have computed $F_{\fp,r.\ell}$ for all $\ell \geq 0$, we can recover the measure $\mu_{\fp,r}$ by computing the series (\ref{Chebyshev polynomial series expansion}). For this purpose we prepare the following lemma:

\begin{lemma}\label{generating series}
    For $|x|<2$ and $|t|<1$ we have 
    \begin{align*}
        \sum_{\ell=0}^\infty T_{\ell}(x) t^{\ell} &= \frac{1}{t^2-xt+1}, \\
        \sum_{\ell=0}^\infty \ell \cdot T_{\ell}(x) t^{\ell} &= \frac{-2t^2+xt}{(t^2-xt+1)^2}, \\
        \sum_{\ell=0}^\infty \ell^2 \cdot T_{\ell}(x) t^{\ell} &= \frac{4t^4-3xt^3+(x^2-4)t^2+xt}{(t^2-xt+1)^3},\\
        \sum_{\ell=0}^\infty \ell^3 \cdot T_{\ell}(x) t^{\ell} &= \frac{t(8t-32t^3+8t^5-x+18t^2x-5t^4x-4tx^2+4t^3x^2-t^2x^3)}{(t^2-xt+1)^4}.\\
    \end{align*}
\end{lemma}
\begin{proof}
	The first equation is standard. To get the rest we apply the relation: for $j\geq0$
	\[t\frac{\,d}{\,dt}\left(\sum_{\ell=0}^\infty \ell^{j} T_{\ell}(x) t^{\ell}\right)=\sum_{\ell=0}^\infty \ell^{j+1} T_{\ell}(x) t^{\ell}.\]
\end{proof}

\subsection{Proof of Theorem \ref{thm. sato tate, with harmonic weight}}
In this subsection we find $\mu_{\fp,2}$ and $\mu_{\fp,3}$.

\bigskip

\noindent\textbf{The case $r=2.$} Set $\fn=\fp^{\ell}, \cO_F$ and apply Corollary \ref{cor. second moment with harmonic weight}, we have 

$$\lim_{\|\tk\|N(\fq) \to \infty}\frac{M_2(\tk,\fq,\fp^\ell)}{M_2(\tk,\fq,\cO_F)} = (\ell+1)N(\fp)^{-\frac{\ell}{2}}.$$
By Lemma \ref{lem, main lem for sato-tate with harmonic weight} and Lemma \ref{generating series}, the measure $\mu_{\fp,2}$ is then given by
\begin{align*}
    d\mu_{\fp,2}(x) &= \sum_{\ell=0}^\infty T_\ell(x) (\ell+1)N(\fp)^{-\frac{\ell}{2}} d\mu_\infty(x) \\
    &= \frac{N(\fp)-1}{(N(\fp)^\frac{1}{2}+N(\fp)^{-\frac{1}{2}}-x)^2}d\mu_\infty(x) \\
    &= \frac{N(\fp)(N(\fp)-1)}{(N(\fp)+1-\sqrt{N(\fp)}x)^2}d\mu_\infty(x).
\end{align*}

\bigskip

\noindent\textbf{The case $r=3.$}  Next we find $\mu_{\fp,3}$. The following is a corollary of Proposition \ref{cubic moment}.

\begin{prop}\label{cor. cubic moment fixed ideal with harmonic weight}
Let $\fn$ be a fixed ideal satisfying $(\fn,\fq)=1$. Then
\[M_3(\mathbf{k},\fq,\fn)=\frac{C_2(\mathbf{k},\fq)\gamma_{-1}^2}{6}\frac{\log^3 N(\fq)}{N(\fn)^{1/2}}\sum_{\fc|\fn} \sum_{\fd|\fn\fc^{-1}} \frac{\mu(\fd)\tau(\fn(\fc\fd)^{-1})}{N(\fd)}+O(N(\fq)\log^2N(\fq)).\]	
\end{prop}

\begin{proof}
By Proposition \ref{cubic moment}, it suffices to study $S_1(\fq,\fn)$ and $S_1(\fq,\fn).$ Since $\fn$ is fixed, we have:
	\[ S_1(\fq,\fn) = \frac{\gamma_{-1}^2}{2} \cdot N(\fn)^{-\frac{1}{2}}\sum_{\fc|\fn} \sum_{\fd|\fn\fc^{-1}} \frac{\mu(\fd)\tau(\fn(\fc\fd)^{-1})}{N(\fd)}\left(\frac{\log N(\fq)}{2}\right)^2+O(\log N(\fq))\]
	and
	\[ S_2(\fq,\fn) = \frac{\gamma_{-1}^2}{3} \cdot N(\fn)^{-\frac{1}{2}}\sum_{\fc|\fn} \sum_{\fd|\fn\fc^{-1}} \frac{\mu(\fd)\tau(\fn(\fc\fd)^{-1})}{N(\fd)}\left(\frac{\log N(\fq)}{2}\right)^3+O(\log^2 N(\fq)).\]
Then apply Proposition \ref{cubic moment} and we complete the proof.
\end{proof}

By Proposition \ref{cor. cubic moment fixed ideal with harmonic weight}, and take $\fn = \fp^\ell$, we have
\begin{align*}
    M_3(\tk,\fq,\fp^\ell) &= \frac{1}{3}D_F^\frac{3}{2}\gamma_{-1}^3\prod_{v|\infty} \frac{k_v-1}{4\pi^2}\frac{(\frac{1}{2}-\frac{1}{2N(\fp)})\ell^2 + (\frac{3}{2}-\frac{1}{2N(\fp)})\ell + 1}{N(\fp)^\frac{\ell}{2}}N(\fq)\log^3N(\fq) + O(N(\fq)\log^2N(\fq)).
\end{align*}
Here the implied constants depend on $\tk,\fp,\ell$ and $F$. Therefore we have
$$\lim_{N(\fq) \to \infty} \frac{M_3(\tk,\fq,\fp^\ell)}{M_3(\tk,\fq,\cO_F)} = \frac{(\frac{1}{2}-\frac{1}{2N(\fp)})\ell^2 + (\frac{3}{2}-\frac{1}{2N(\fp)})\ell + 1}{N(\fp)^\frac{\ell}{2}}.$$
The measure $\mu_{\fp,3}$ is then given by
\begin{align*}
    d\mu_{\fp,3}(x) &= \left(\frac{N(\fp)-1}{2N(\fp)}\sum_{\ell=0}^\infty \ell^2T_\ell(x)N(\fp)^{-\frac{\ell}{2}} + \frac{3N(\fp)-1}{2N(\fp)}\sum_{\ell=0}^\infty \ell T_\ell(x)N(\fp)^{-\frac{\ell}{2}} + \sum_{\ell=0}^\infty T_\ell(x)N(\fp)^{-\frac{\ell}{2}}\right)d\mu_\infty(x) \\
    &= \left(\frac{N(\fp)-1}{N(\fp)+1-\sqrt{N(\fp)}x}\right)^3d\mu_\infty(x).
\end{align*}
via Lemma \ref{generating series}.

\subsection{Proof of Theorem \ref{thm. sato tate, without harmonic weight}}  In this subsection, we find $\mu_{\fp,2}^{un}$ and $\mu_{\fp,3}^{un}.$

\bigskip

\noindent\textbf{The case $r=2.$} Apply Proposition \ref{prop. second moment without harmonic weight}, and we prove the following result.
\begin{cor}\label{cor. second moment fixed ideal}
Let $\fp$ be a fixed ideal and $\ell\geq0$ a fixed integer. Then we have    
 \begin{flalign*}
  	\frac{\cM_2(\tk,\fq,\fp^{\ell})}{\cM_2(\mathbf{k},\fq,\cO_F)}&=\frac{1}{N(\fp)^{\frac{\ell}{2}}}\left(\ell+1+\sum_{\alpha=1}^{\ell}(\ell+1-\alpha)\left(1-\frac{1}{N(\fp)}\right)\left\{\frac{1+(-1)^{\alpha}}{2}+\frac{(1-(-1)^{\alpha})N(\fp)}{N(\fp)^2+1}\right\}\right)+O\left(\frac{1}{\log Q_2}\right)
  \end{flalign*}
\end{cor}
\begin{proof}
	Notice that $\fn=\fp^{\ell}$ is fixed, then the terms coming form $\log N(\fn\fl^2\fd^{-2})$ in Proposition \ref{prop. second moment without harmonic weight} will only contribute the error term. Therefore:
 \begin{flalign*}
  	\cM_2(\tk,\fq,\fp^{\ell})&=\frac{C_2(\mathbf{k},\fq)\log Q_2}{N(\fp)^{\frac{\ell}{2}}}\sum_{\fr|\fp^{\ell}}\tau(\fp^{\ell}\fr^{-1})\sum_{\substack{N(\fl\ff^2)\ll Q_2^{\varepsilon} \\ (\fl\ff,\fq)=1,\fr|\fl^2}}\frac{\tau(\fl^2\fr^{-1})}{N(\fl)^2N(\ff)^2}\sum_{\fd|\fr}N(\fd)\mu(\fr\fd^{-1})\\
  	&\hspace{20mm}+O(\|\mathbf{k}\|N(\fq))+O((\|\mathbf{k}\|N(\fq))^{1-\eta'})
\end{flalign*}	
This can be further written as:
 \begin{flalign*}
  	\cM_2(\tk,\fq,\fp^{\ell})&=\frac{C_2(\mathbf{k},\fq)\log Q_2}{N(\fp)^{\frac{\ell}{2}}}\sum_{\fr|\fp^{\ell}}\tau(\fp^{\ell}\fr^{-1})\sum_{\fd|\fr}\frac{\mu(\fd)}{N(\fd)}\sum_{\substack{(\fl\ff,\fq)=1,\\ \fr|\fl^2}}\frac{N(\fr)\tau(\fl^2\fr^{-1})}{N(\fl)^2N(\ff)^2}\\
  	&\hspace{20mm}+O(\|\mathbf{k}\|N(\fq))+O((\|\mathbf{k}\|N(\fq))^{1-\eta'}).
  \end{flalign*}
For the last sum, we have:
\[\sum_{\substack{(\fl\ff,\fq)=1,\\ \fr|\fl^2}}\frac{N(\fr)\tau(\fl^2\fr^{-1})}{N(\fl)^2N(\ff)^2}=(\zeta_F^{(\fq)}(2))^3\prod_{\fp\neq\fq}\left(1+\frac{1}{N(\fp)^2}\right)\times\begin{cases}
	1&\mbox{if $e_{\fp}(\fr)$ even}\\
	\frac{2N(\fp)}{N(\fp)^2+1}&\mbox{if $e_{\fp}(\fr)$ odd}
\end{cases}.\]
Insert it into $\mathcal{M}_2(\mathbf{k},\fq,\fp^{\ell})$, and we complete the proof.
\end{proof}

Set $\fn=\fp^{\ell},\cO_F$, and apply Corollary \ref{cor. second moment fixed ideal}:
\[\lim_{\|\tk\|N(\fq) \to \infty}\frac{\cM_2(\tk,\fq,\fp^\ell)}{\cM_2(\tk,\fq,\cO_F)}=\frac{1}{N(\fp)^{\frac{\ell}{2}}}\left(\ell+1+\sum_{\alpha=1}^{\ell}(\ell+1-\alpha)\left(1-\frac{1}{N(\fp)}\right)\left\{\frac{1+(-1)^{\alpha}}{2}+\frac{(1-(-1)^{\alpha})N(\fp)}{N(\fp)^2+1}\right\}\right).\]
Apply Lemma \ref{lem, main lem for sato-tate with harmonic weight} and Lemma \ref{generating series}, and we conclude:
\begin{flalign*}
\,d\mu_{\fp,2}^{un}(x)&=\left(\frac{(N(\fp)-1)(N(\fp)+1)^2}{4N(\fp)(N(\fp)^2+1)}\sum_{\ell=0}^\infty\frac{\ell^2 T_{\ell}(x)}{N(\fp)^{\frac{\ell}{2}}}+\frac{N(\fp)^2+N(\fp)}{N(\fp)^2+1}\sum_{\ell=0}^\infty\frac{\ell T_{\ell}(x)}{N(\fp)^{\frac{\ell}{2}}}\right.\\
&\hspace{10mm}\left.+\left(1-\frac{(N(\fp)-1)^3}{8N(\fp)(N(\fp)^2+1)}\right)\sum_{\ell=0}^\infty\frac{T_{\ell}(x)}{N(\fp)^{\frac{\ell}{2}}}+\frac{(N(\fp)-1)^3}{8N(\fp)(N(\fp)^2+1)}\sum_{\ell=0}^\infty\frac{(-1)^{\ell} T_{\ell}(x)}{N(\fp)^{\frac{\ell}{2}}}\right)\,d\mu_{\infty}(x)\\
&=\frac{(1-N(\fp)^{-2})^3}{(1+N(\fp)^{-2})}\frac{1}{\left(1-\frac{x}{N(\fp)^{1/2}}+\frac{1}{N(\fp)}\right)^3\left(1+\frac{x}{N(\fp)^{1/2}}+\frac{1}{N(\fp)}\right)}\,d\mu_{\infty}(x).
\end{flalign*}

\bigskip

\noindent\textbf{The case $r=3.$} Before finding $\mu_{\fp,3}^{un}$, we need some preparation work. We define the following function on integral ideals $\fm \subset \cO_F$:
\[g(\fm):=\sum_{\fc|\fm} \tau(\fm\fc^{-1})\sum_{\fd|\fc} \frac{\mu(\fd)}{N(\fd)}.\]
We can prove the following lemma:
\begin{lemma}
For $g(\fm),$  we have the following properties:
\begin{enumerate}\label{lem. properties of g(m)}
    \item  $g(\fm)$ is a multiplicative function. 
    \item  $g(\fp^{\ell})=\ell+1+\left(1-\frac{1}{N(\fp)}\right)\frac{\ell(\ell+1)}{2}.$   
\end{enumerate}
\end{lemma}
We set 
\[L(s,g):=\sum_{\fm\subseteq\cO_F}\frac{g(\fm)}{N(\fm)^{s}}.\]
This is absolutely convergent for $\Re(s)>1$ and has an Euler product expansion:
\[L(s,g)=\prod_{\fp}\left(\sum_{\ell\geq0}\frac{g(\fp^{\ell})}{N(\fp)^{\ell s}}\right)=\prod_{\fp}L_{\fp}(s,g).\]
We also use the standard notation: for a set $S$ consisting of finite places, we set
\[L^{(S)}(s,g)=\prod_{\fp\notin S}L_{\fp}(s,g)\hspace{8mm}\mbox{and}\hspace{8mm}L_S(s,g)=\prod_{\fp \in  S}L_{\fp}(s,g).\]

The next case is the cubic moment without the harmonic weight: 
\begin{prop}\label{cor. cubic moment fixed ideal without the harmonic weight}
Assume the notations above. Let $\fn=\fp^{\ell}$ be a fixed ideal satisfying $(\fp,\fq)=1$. Then
\[\frac{\cM_3(\mathbf{k},\fq,\fn)}{\cM_{3}(\mathbf{k},\fq,\cO_F)}=\frac{1}{N(\fp)^{\ell/2}}\frac{H(\fp^{\ell})}{H(\cO_F)}+O\left(\frac{1}{\log N(\fq)}\right).\]	
Here $H(\fp^{\ell})$ is defined in \eqref{cubic moment, main term,  with harnomic weight}. Moreover, an explicit expression for $\frac{H(\fp^{\ell})}{H(\cO_F)}$ is given in Appendix \ref{apx. sato tate cubic without}.
\end{prop}
\begin{proof}
	Let $\fn$ be a fixed ideal and $\fl$ be an ideal satisfying $N(\fl)\ll N(\fq)^{\varepsilon}.$ Let $\fr|(\fn,\fl^2)$ and set $\fm=\fn\fl^2\fr^{-2}$.  Then we have 
\begin{equation}\label{eq. controlled s1}
		S_1(\fq,\fm)=\frac{\gamma_{-1}^2\log N(\fq)^2}{8}\frac{1}{N(\fm)^{1/2}}\sum_{\fc|\fn} \tau(\fn\fc^{-1})\sum_{\fd|\fc} \frac{\mu(\fd)}{N(\fd)}+O\left(\frac{\log N(\fq)}{N(\fl)^{1-\varepsilon}}\right)+O(N(\fq)^{-1/2+\varepsilon}N(\fl)^{1/2+\varepsilon})
\end{equation}
	and
\begin{equation}\label{eq. controlled s2}
S_2(\fq,\fm)=\frac{\gamma_{-1}^2\log N(\fq)^3}{24}\frac{1}{N(\fm)^{1/2}}\sum_{\fc|\fn} \tau(\fn\fc^{-1})\sum_{\fd|\fc} \frac{\mu(\fd)}{N(\fd)}+O\left(\frac{\log^2 N(\fq)}{N(\fl)^{1-\varepsilon}}\right)+O(N(\fq)^{-1/2+\varepsilon}N(\fl)^{1/2+\varepsilon}).
\end{equation}
Next, by Proposition \ref{cubic moment} and the convexity bound (in the level aspect) for $L(1/2,\pi),$ $\alpha_{\pi}=L(1/2,\pi)^3 \lambda_\pi(\fn)$ satisfies condition \eqref{conditions to estimate the middle sum}. Therefore,
\[\cM_3(\mathbf{k},\fq,\fn)=\sum_{\substack{N(\fl\ff^2)\ll N(\fq)^{\varepsilon} \\ (\fl\ff,\fq)=1}}\frac{1}{N(\fl\ff^2)}\sum_{\fr|(\fn,\fl^2)}M_3(\mathbf{k},\fq,\fn\fl^2\fr^{-2})+O((\|\mathbf{k}\|N(\fq))^{1-\eta'}).\] 	
Therefore, by Proposition \ref{cubic moment}, \eqref{eq. controlled s1} and \eqref{eq. controlled s2},
\begin{equation}\label{eq. cubic moment, without harmonic weight, fixed ideal}
\cM_3(\mathbf{k},\fq,\fn)=\frac{C_2(\mathbf{k},\fq)\gamma_{-1}^2\log N(\fq)^3}{6}G(\fn)+O(\log^2N(\fq)N(\fq)),
\end{equation}
where
\begin{flalign*}
G(\fn)&=\sum_{\substack{N(\fl\ff^2)\ll N(\fq)^{\varepsilon} \\ (\fl\ff,\fq)=1}}\frac{1}{N(\fl\ff^2)}\sum_{\fr|(\fn,\fl^2)}\frac{N(\fr)}{N(\fn)^{1/2}N(\fl)}\sum_{\fc|\fn\fl^2\fr^{-2}}\tau(\fn\fl^2\fr^{-2}\fc^{-1}) \sum_{\fd|\fc} \frac{\mu(\fd)}{N(\fd)}\\
&=\frac{1}{N(\fn)^{1/2}}\zeta_F^{(\fq)}(2)\sum_{\fr|\fn}\sum_{\substack{\fr|\fl^2 \\ (\fl,\fq)=1}}\frac{N(\fr)}{N(\fl)^2}\sum_{\fc|\fn\fl^2\fr^{-2}} \tau(\fn\fl^2\fr^{-2}\fc^{-1})\sum_{\fd|\fc} \frac{\mu(\fd)}{N(\fd)}+O(N(\fq)^{-\varepsilon})\\
&=\frac{1}{N(\fn)^{1/2}}\zeta_F^{(\fq)}(2)\sum_{\fr|\fn}\sum_{\substack{\fr|\fl^2 \\ (\fl,\fq)=1}}\frac{N(\fr)}{N(\fl)^2}g(\fn\fl^2\fr^{-2})+O(N(\fq)^{-\varepsilon}).
\end{flalign*}
\begin{remark}\label{rem. cubic moment with random matrix theory}
	Set $\fm=\cO_F$ and this becomes
\[\sum_{\substack{\pi\in \mathcal{F}(\mathbf{k},\mathfrak{q})}}\frac{L(1/2,\pi)^3}{L(1,\pi,\Ad)}\sim \frac{C_2(\mathbf{k},\fq)\gamma_{-1}^2\log^3 N(\fq)}{6}.\]
This coincides with the first main term predicted by the random matrix theory. 
\end{remark}

Next, we take $\fn=\fp^{\ell}.$ Then by the multiplicativity of $g(\fm)$:
\[G(\fn)=\frac{\zeta_F^{(\fq)}(2)L^{(\fq\fp)}(2,g)}{N(\fp)^{\frac{\ell}{2}}}\left(\sum_{\alpha=0}^{\ell}\sum_{\substack{r\geq0\\ 2r\geq\alpha}}\frac{N(\fp^{\alpha})g(\fp^{\ell+2r-2\alpha})}{N(\fp)^{2r}}\right):=\frac{\zeta_F^{(\fq)}(2)L^{(\fq\fp)}(2,g)}{N(\fp)^{\frac{\ell}{2}}} H(\fp^{\ell}).\]
Depending on the parity of $\alpha,$ we have
\[H(\fp^{\ell})=\sum_{\substack{r\geq0\\ 2r\geq\alpha}}\frac{N(\fp)^{\alpha}g(\fp^{\ell+2r-2\alpha})}{N(\fp)^{2r}}=\begin{cases}
\sum\limits_{r\geq0}\frac{g(\fp^{\ell-\alpha+2r})}{N(\fp)^{2r}}&\mbox{if $\alpha$ is even}	\\
\frac{1}{N(\fp)}\sum\limits_{r\geq0}\frac{g(\fp^{\ell-\alpha+2r+1})}{N(\fp)^{2r}}&\mbox{if $\alpha$ is odd}	
\end{cases}.\]
Therefore,
\begin{equation}\label{cubic moment, main term,  with harnomic weight}
H(\fp^{\ell})=\sum_{\alpha=0}^{\ell}\frac{1+(-1)^{\ell}}{2}\sum_{\beta\geq0}\frac{g(p^{\ell-\alpha+2\beta})}{N(\fp)^{2\beta}}+\frac{1+(-1)^{\ell}}{2N(\fp)}\sum_{\beta\geq0}\frac{g(p^{\ell-\alpha+2\beta+1})}{N(\fp)^{2\beta}},
\end{equation}
and one can show that $H(\cO_F)=L_{\fp}(2,g)$. This will imply:
\[\frac{\cM_3(\mathbf{k},\fq,\fp^{\ell})}{\cM_3(\mathbf{k},\fq,\cO_F)}=\frac{1}{N(\fp)^{\ell/2}}\frac{\zeta_F^{(\fq)}(2)L^{(\fp\fq)}(2,g)}{\zeta_F^{(\fq)}(2)L^{(\fq)}(2,g)}+O\left(\frac{1}{\log N(\fq)}\right)=\frac{1}{N(\fp)^{\ell/2}}\frac{H(\fp^{\ell})}{H(\cO_F)}+O\left(\frac{1}{\log N(\fq)}\right).\]
\end{proof}

Finally, we find $\mu_{\fp,3}^{un}:$ this is a direct corollary of Lemma \ref{lem. properties of g(m)}, Proposition \ref{cubic moment, main term,  with harnomic weight} and Lemma \ref{generating series}. The final expression is:  
\[\,d\mu_{\fp,3}^{un}(x)=\frac{(1-N(\fp)^{-1})^5(1+N(\fp)^{-1})^4}{(1+N(\fp)^{-1}+4N(\fp)^{-2}+N(\fp)^{-3}+N(\fp)^{-4})}\frac{1}{\left(1-\frac{x}{N(\fp)^{1/2}}+\frac{1}{N(\fp)}\right)^4\left(1+\frac{x}{N(\fp)^{1/2}}+\frac{1}{N(\fp)}\right)}\,d\mu_{\infty}(x).\]

\section{Weighted Low-lying Zeros}\label{sec weighted low lying zero}

Let $\Phi$ be a test function. Recall that we have defined the 1-level density of low-lying zeros $D(\pi;\Phi)$ as in \eqref{1-level density}, with $c_\pi = \|\tk\|^2N(\fq)$. We set  
\begin{align*}
A_r(\tk,\fq;\Phi) &:= \frac{1}{\sum_{\pi \in \cF(\tk,\fq)}\frac{L(1/2,\pi)^r}{L(1,\pi,\Ad)}}\sum_{\pi \in \cF(\tk,\fq)}\frac{L(1/2,\pi)^r}{L(1,\pi,\Ad)}D(\pi;\Phi), \\
B_r(\tk,\fq;\Phi) &:= \frac{1}{\sum_{\pi \in \cF(\tk,\fq)}{L(1/2,\pi)^r}}\sum_{\pi \in \cF(\tk,\fq)}L(1/2,\pi)^rD(\pi;\Phi). 
\end{align*}

Our goal in this section is to find distributions $W_r$ on $\bR$, for $r=1,2,3$, such that when supp$(\hat{\Phi})$ is restricted to a certain range, we have  
\begin{align*}
    \lim_{Q_r\to \infty} A_r(\tk,\fq;\Phi) = \lim_{Q_r\to \infty} B_r(\tk,\fq;\Phi) = \int_\bR \Phi(x)W_r(x)dx.
\end{align*}
We will only prove the case $r=2$ and $r=3$, since the proof of the $r=1$ case is easier.

\subsection{Explicit Formula}\label{subsec, explicit formula}
By a standard explicit formula argument (see, for example, \cite{IwaniecLuoSarnak2000}), we have the following expression of $D(\pi;\Phi)$:
\begin{align*}
D(\pi;\Phi) &= \hat{\Phi}(0) + \frac{1}{2}\Phi(0) + O\left(\frac{\log\log 3N(\fq)}{\log R}\right)\\
&- \frac{1}{\log R}\sum_{\fp \neq \fq} \frac{2\log N(\fp)}{N(\fp)^\frac{1}{2}}\lambda_\pi(\fp)\hat{\Phi}\left(\frac{\log N(\fp)}{\log R}\right) \\
&- \frac{1}{\log R}\sum_{\fp \neq \fq} \frac{2\log N(\fp)}{N(\fp)}\lambda_\pi(\fp^2)\hat{\Phi}\left(\frac{2\log N(\fp)}{\log R}\right) .
\end{align*}
Here we define $R := \|\tk\|^2N(\fq)$ and the implied constant depends on $\Phi$ and $F$. 

To deal with the summation over $\fp$ we prepare the following lemma.
\begin{lemma}\label{Landau prime ideal theorem with test function}
     Let $n\geq1$ be a positive integer. Let $\Phi$ be a test function. Let $\fN$ be a fixed square-free ideal and $R\geq1$ a large number. Then
     \begin{equation}\label{prime ideal theorem}
    \frac{1}{\log R}\sum_{\fp\nmid \fN}\frac{(\log N(\fp))^n}{N(\fp)}\widehat{\Phi}\left(\frac{\log N(\fp)}{\log R}\right)=(\log R)^{n-1}\int_0^{\infty}\widehat{\Phi}(u)u^{n-1}\,du+O((\log R)^{n-2}).
     \end{equation}
\end{lemma}
\begin{proof}
    Denote by $\Pi(x)$ the prime ideal counting function over $F$. Then 
    \[\Pi(x)=\frac{x}{\log x}+O\left(\frac{x}{(\log x)^A}\right)\]
  for any $A>0.$ 
 Then summation by parts implies:
 \[\sum_{\fp}\frac{(\log N(\fp))^n}{N(\fp)}\hat{\Phi}\left(\frac{\log N(\fp)}{\log R}\right)=-\int_2^{\infty}\Pi(x)\left(\frac{(\log x)^n}{x}\hat{\Phi}\left(\frac{\log x}{\log R}\right)\right)'\,dx\]
By the asymptotic formula for $\Pi(x),$ it suffices to investigate
\[\int_2^{\infty}\frac{x}{\log x}\left(\frac{(\log x)^n}{x}\hat{\Phi}\left(\frac{\log x}{\log R}\right)\right)'\,dx.\]
Integration by parts implies:
\[\int_2^{\infty}\frac{x}{\log x}\left(\frac{(\log x)^n}{x}\hat{\Phi}\left(\frac{\log x}{\log R}\right)\right)'\,dx=-\int_2^{\infty}\frac{\log x-1}{(\log x)^2}\frac{(\log x)^n}{x}\hat{\Phi}\left(\frac{\log x}{\log R}\right)\,dx+O(1).\]
It suffices to study:
\[\int_2^{\infty}\frac{\log x}{(\log x)^2}\frac{(\log x)^n}{x}\hat{\Phi}\left(\frac{\log x}{\log R}\right)\,dx=\int_2^{\infty}\frac{(\log x)^{n-1}}{x}\hat{\Phi}\left(\frac{\log x}{\log R}\right)\,dx\]
Set $u=\frac{\log x}{\log R}$. Then $\,du=\frac{\,dx}{x\log R}$ and 
\[\int_2^{\infty}\frac{(\log x)^{n-1}}{x}\hat{\Phi}\left(\frac{\log x}{\log R}\right)\,dx=(\log R)^n\int_{\frac{\log 2}{\log R}}^{\infty}\hat{\Phi}(u)u^{n-1}\,du=(\log R)^n\int_{0}^{\infty}\hat{\Phi}(u)u^{n-1}\,du+O(1).\]
This shows:
\[\sum_{\fp}\frac{(\log N(\fp))^n}{N(\fp)}\hat{\Phi}\left(\frac{\log N(\fp)}{\log R}\right)=(\log R)^n\int_{0}^{\infty}\hat{\Phi}(u)u^{n-1}\,du+O((\log R)^{n-1}),\]
which will imply \eqref{Landau prime ideal theorem with test function}.
\end{proof}

Finally, we prove the following lemma, which is helpful in removing harmonic weight.
\begin{lemma}\label{lem. llz, remove harmonic weight useful lemma}
Set $R = \|\tk\|^2N(\fq)$. Then for some $A>0,$ we have:
    \[D(\pi;\Phi)\ll\log^A R.\]
This implies
    \begin{equation}\label{eq. estimation for the key term in 1 level density}
   \sum_{\fp \neq \fq} \frac{2\lambda_\pi(\fp)\log N(\fp)}{N(\fp)^\frac{1}{2}}\hat{\Phi}\left(\frac{\log N(\fp)}{\log R}\right)\ll\log^A R
    \end{equation}
    for some $A>0.$	
\end{lemma}
\begin{proof}
 By the definition of $D(\pi;\Phi),$ one has:
 \[D(\pi;\Phi)\ll\sum_{|\gamma_{\pi}|\leq1}\left|\hat{\Phi}\left(\frac{\gamma_{\pi}\log \|\textbf{k}\|^2N(\fq)}{2\pi}\right)\right|+\sum_{A\geq1}\sum_{A\leq|\gamma_{\pi}|\leq A+1}\left|\hat{\Phi}\left(\frac{\gamma_{\pi}\log \|\textbf{k}\|^2N(\fq)}{2\pi}\right)\right|.\]
Then \cite[Proposition 5.7]{IwaniecKowalski2004} and the fact that $\hat{\Phi}$ is a compactly supported function imply the bound.

Notice that in the explicit formula of $D(\pi;\Phi)$, all terms except for \eqref{eq. estimation for the key term in 1 level density} are bounded by $\log^A R.$ So \eqref{eq. estimation for the key term in 1 level density} is valid. 
\end{proof}

\subsection{Fourier Transforms}\label{subsec. fourier transformation} 
In this subsection we collect some Fourier transform formulas for later use. Let $\Phi$ be a test function with supp$(\hat{\Phi}) \subset (-1,1)$. Let $n \geq 0$ be a natural number. Then there exists distribution $h_n(x)$ on $\bR$ defined by 
$$h_n(x) := \int_0^1 t^n \cos(2\pi xt) dt,$$
such that 
$$\int_0^\infty \hat{\Phi}(u)|u|^ndu = \int_{\infty}^\infty \Phi(x)h_n(x)dx.$$
Explicit formulas of $h_n(x)$ for $n \leq 3$ are given as follows:
\begin{align*}
    h_0(x) &= \frac{\sin(2\pi x)}{2\pi x}, \\
    h_1(x) &= \frac{\sin(2\pi x)}{2\pi x} + \frac{\cos(2\pi x)-1}{(2\pi x)^2}, \\
    h_2(x) &= -\frac{\cos(2\pi x)}{2\pi x} + \frac{\sin(2\pi x)}{(2\pi x)^2} - \frac{\cos(4\pi x)-2}{(2\pi x)^3} \\
    h_3(x) &=\frac{\sin(2\pi x)}{2\pi x}+\frac{ 3\cos(2\pi x)}{4\pi^2 x^2} -\frac{3(2\pi x\sin(2\pi x)+(\cos(2\pi x)-1))}{8\pi^4x^4}.
\end{align*}

\subsection{Proof of Theorem \ref{thm. llz, with harmonic weight}} 
By the explicit formula for $D(\pi;\Phi)$, we have the following explicit formula for the weighted average $A_r(\tk,\fq;\Phi)$:
\begin{align*}
    A_r(\tk,\fq;\Phi) &= \hat{\Phi}(0) + \frac{1}{2}\Phi(0) + O\left(\frac{\log\log 3N(\fq)}{\log Q_2}\right) \\
    &\hspace{10mm}- \frac{1}{\log Q_2}\sum_{\fp \neq \fq} \frac{2\log N(\fp)}{N(\fp)^\frac{1}{2}} \frac{M_r(\tk,\fq,\fp)}{M_r(\tk,\fq,\cO_F)}\hat{\Phi}\left(\frac{\log N(\fp)}{\log Q_2}\right) \\
    &\hspace{10mm}- \frac{1}{\log Q_2}\sum_{\fp \neq \fq} \frac{2\log N(\fp)}{N(\fp)} \frac{M_r(\tk,\fq,\fp^2)}{M_r(\tk,\fq,\cO_F)}\hat{\Phi}\left(\frac{2\log N(\fp)}{\log Q_2}\right)\\
    &=\hat{\Phi}(0) + \frac{1}{2}\Phi(0)-E_r^{(1)}(\tk,\fq;\Phi)- E_r^{(2)}(\tk,\fq;\Phi)+ O\left(\frac{\log\log 3N(\fq)}{\log Q_2}\right)
\end{align*}
where we set
\begin{align*}
    E_r^{(1)}(\tk,\fq;\Phi) &:= \frac{1}{M_r(\tk,\fq,\cO_F) \log Q_2}\sum_{\fp \neq \fq} \frac{2\log N(\fp)}{N(\fp)^\frac{1}{2}} M_r(\tk,\fq,\fp)\hat{\Phi}\left(\frac{\log N(\fp)}{\log Q_2}\right), \\
    E_r^{(2)}(\tk,\fq;\Phi) &:= \frac{1}{M_r(\tk,\fq,\cO_F) \log Q_2}\sum_{\fp \neq \fq} \frac{2\log N(\fp)}{N(\fp)} M_r(\tk,\fq,\fp^2)\hat{\Phi}\left(\frac{2\log N(\fp)}{\log Q_2}\right).
\end{align*}

\bigskip

\noindent\textbf{The case $r=2.$}  Assume $\fq$ is prime. Suppose $\hat{\Phi}$ is supported in $(-\alpha_2,\alpha_2)$. As special cases of Corollary \ref{cor. second moment with harmonic weight} we have 
\begin{align*}
    M_2(\tk,\fq,\cO_F) &= C_2(\mathbf{k},\fq)(\log Q_2 +O(1))+ O(\|\tk\|^{\frac{1}{2}+\varepsilon}N(\fq)^\varepsilon) \\
    M_2(\tk,\fq,\fp) &= C_2(\mathbf{k},\fq)\frac{2}{N(\fp)^{1/2}}(\log Q_2-\log N(\fp)+O(1)) + O(\|\tk\|^{\frac{1}{2}+\varepsilon}N(\fq)^\varepsilon N(\fp)^{\frac{1}{2}+\varepsilon}) \\
    M_2(\tk,\fq,\fp^2) &= C_2(\mathbf{k},\fq)\frac{3}{N(\fp)}(\log Q_2-2\log N(\fp)+O(1)) + O(\|\tk\|^{\frac{1}{2}+\varepsilon}N(\fq)^\varepsilon N(\fp)^{1+\varepsilon}).
\end{align*}
Then by Lemma \ref{prime ideal theorem} and the definition of $E_2^{(1)}(\tk,\fq;\Phi)$,
$$\lim_{\|\tk\|N(\fq) \to \infty}E_2^{(1)}(\tk,\fq;\Phi) =4\int_{0}^\infty\hat{\Phi}(u)\,du-4\int_{0}^\infty\hat{\Phi}(u)u\,du = -4\int_{0}^\infty\hat{\Phi}(u)u\,du + 2\Phi(0)$$
provided that $\alpha_2 < \frac{1}{4}$.

Similarly, for $E_2^{(2)}(\tk,\fq;\Phi)$ we have 
$$\lim_{\|\tk\|N(\fq) \to \infty} E_2^{(2)}(\tk,\fq;\Phi) = 0$$
as $\alpha_2<\frac{1}{4}$. Therefore we conclude
$$\lim_{\|\tk\|N(\fq) \to \infty}A_2(\tk,\fq;\Phi) = \hat{\Phi}(0) - \frac{3}{2}\Phi(0) + 4\int_0^\infty \hat{\Phi}(u)u\,du$$
as supp$(\hat{\Phi}) \subset(-\frac{1}{4},\frac{1}{4}).$ Then apply the Fourier transformations in \S \ref{subsec. fourier transformation} and we complete the proof for $r=2$.
\qed

\begin{remark}
If we fix $\tk$ and taking the limit as $N(\fq) \to \infty$, then we need only to assume $\alpha_2 < \frac{1}{2}$.
\end{remark}

\bigskip

\noindent\textbf{The case $r=3.$} Fix a weight $\tk$ and consider $N(\fq)$ being large. Suppose $\hat{\Phi}$ is supported in $(-\alpha_3,\alpha_3)$. We prove the following useful lemma:
\begin{lemma}\label{lem. useful lemma when removing the harmonic weight in llz}
    Assume that $N(\fn)\ll N(\fq)^{1/2+\varepsilon}.$ Then
    \[M_{3}(\tk,\fq,\fn)=O\left(\frac{\tau(\fn)^2\log^3 N(\fq)}{N(\fn)^{1/2}}N(\fq)\right)+O(N(\fq)^{1/2+\varepsilon}N(\fn)^{1/2+\varepsilon}).\]
\end{lemma}
\begin{proof}
Apply Proposition \ref{cubic moment}, and we obtain:
\begin{flalign*}
M_3(\tk,\fq,\fn)=2C_2(\tk,\fq)\{(\log N(\fq)+c_0(k))S_1(\fq,\fn)-S_2(\fq,\fn)\}+O(N(\fq)^{1/2+\varepsilon}N(\fn)^{1/2+\varepsilon}).\\
\end{flalign*} 
By checking the form of $S_1(\fq,\fn)$ and $S_2(\fq,\fn),$ we complete the proof.
\end{proof} 
Next, we state the following proposition for the weighted cubic moment.
\begin{prop}\label{cor. cubic moment prime ideal with harmonic wieght}
	Let $\eta>0$. Let $\fp$ be a prime ideal with $N(\fp)\ll N(\fq)^{1/2-\eta}$. Then
\begin{align*}
M_3(\mathbf{k},\fq,\fp)&=2C_2(\mathbf{k},\fq)\gamma_{-1}^2\frac{\cL(\log Q,\log N(\fp))}{N(\fp)^{1/2}}+O\left(\frac{(\log Q)^2N(\fq)}{N(\fp)^{1/2}}\right)+O\left(\frac{(\log Q)^3N(\fq)}{N(\fp)^{3/2}}\right)+O(N(\fp)^{\frac{1}{2}+\varepsilon}N(\fq)^{\frac{1}{2}+\varepsilon}),
\end{align*}
where $\cL(x,y)\in \bQ[x,y]$ and:
\[\cL(x,y)=\frac{x^3}{4}-\frac{x^2y}{2}+\frac{y^3}{3}.\]
Moreover, for $\fp$ be a prime ideal with $N(\fp)^2\ll N(\fq)^{1/2-\eta},$
\[M_3(\mathbf{k},\fq,\fp^2)=O\left(\frac{\log^3N(\fq)N(\fq)}{N(\fp)}\right)+O(N(\fp)^{\frac{1}{2}+\varepsilon}N(\fq)^{\frac{1}{2}+\varepsilon}).\]
\end{prop}
\begin{proof}
	By Proposition \ref{cubic moment}, it suffices to find $S_1(\fq,\fp)$ and $S_2(\fq,\fp).$ We have:
    \begin{align*}
    S_1(\fq,\fp)
    &=\frac{\gamma_{-1}^2}{2 N(\fp)^{\frac{1}{2}}} \left(\tau(\fp)\left(\frac{\log N(\fq)}{2}\right)^2+ \left(\frac{\log N(\fq)}{2}-\log N(\fp)\right)^2\right)+O\left(\frac{\log N(\fq)}{N(\fp)^{1/2}}\right)+O(N(\fq)^{-1/2+\varepsilon}N(\fp)^{1/2+\varepsilon})\\
\end{align*}
and
\begin{flalign*}
    S_2(\fq,\fp) &= \frac{\gamma_{-1}^2}{3N(\fp)^{1/2}} \left(\tau(\fp)\left(\frac{\log N(\fq)}{2}\right)^3+\left(\frac{\log N(\fq)}{2}-\log N(\fp)\right)^3 \right)+ \frac{\gamma_{-1}^2}{2N(\fp)^{1/2}} \tau(\fp)\log N(\fp)\left(\frac{\log^2N(\fq)}{2}\right)^2 \\
    &\hspace{10mm}+O\left(\frac{\tau(\fl^2)^2\log^3 N(\fl)\log^2 N(\fq)}{N(\fp)^{1/2}}\right)+O(N(\fq)^{-1/2+\varepsilon}N(\fp\fl^2)^{1/2+\varepsilon})
\end{flalign*}
Apply Proposition \ref{cubic moment}, and we complete the proof for the first part.
	
The second part is a corollary of Lemma \ref{lem. useful lemma when removing the harmonic weight in llz}.
\end{proof}

By Remark \ref{rem. cubic moment with random matrix theory} and Proposition \ref{cor. cubic moment prime ideal with harmonic wieght}, we obtain:
\[E_3^{(1)}(\tk,\fq;\Phi)=3\Phi(0)-12\int_0^\infty\ \hat{\Phi}(u)u\,du +8\int_0^\infty \hat{\Phi}(u)u^3\,du+O\left(\frac{1}{\log N(\fq)}\right)+O(N(\fq)^{-1/2+\delta+\varepsilon})\]
Similarly as before we have 
$$\lim_{N(\fq) \to \infty}A_3(\tk,\fq;\Phi) = \hat{\Phi}(0)-\frac{5}{2}\Phi(0)+12\int_0^\infty\ \hat{\Phi}(u)u\,du -8\int_0^\infty \hat{\Phi}(u)u^3\,du$$
as supp$(\hat{\Phi}) \subset (-\frac{1}{2},\frac{1}{2}).$ Then apply the Fourier transformations in \S \ref{subsec. fourier transformation} and we complete the proof.
\qed

\subsection{Proof of Theorem \ref{thm. llz, without harmonic weight}}

For $r=2,3$, set $\alpha_{\pi,r}=L(1/2,\pi)^rD(\pi;\Phi)$. By Lemma \ref{lem. llz, remove harmonic weight useful lemma}, it satisfies the conditions in \eqref{conditions to estimate the middle sum} (notice that $\|\tk\|$ is fixed when $r=3$). Therefore, we have
\[\sum_{\pi\in\cF(\tk,\fq)}L(1/2,\pi)^rD(\pi;\Phi)=\sum_{\pi\in\cF(\tk,\fq)}\omega_{\pi}(x)\frac{L(1/2,\pi)^r}{L(1,\pi,\Ad)}D(\pi;\Phi)+O((\|\tk\|N(\fq))^{1-\eta}).\]
Here $x=Q_r^{\varepsilon}$, $\omega_{\pi}(x)$ is defined in \S \ref{remove the harmonic weight}, and  $\eta$ is dependent on $\varepsilon$.

Next, insert the explicit formula for $D(\pi;\Phi)$,and we obtain:

\begin{equation}\label{explicit formula without harmonic weight, general}
    B_r(\tk,\fq;\Phi)=\hat{\Phi}(0) + \frac{1}{2}\Phi(0)-\cE_r^{(1)}(\tk,\fq;\Phi)- \cE_r^{(2)}(\tk,\fq;\Phi)+ O\left(\frac{\log\log 3N(\fq)}{\log R}\right)
\end{equation}
where
\begin{align*}
    \cE_r^{(1)}(\tk,\fq;\Phi) &:= \frac{1}{\cM_r(\tk,\fq,\cO_F) \log R}\sum_{\fp \neq \fq} \frac{2\log N(\fp)}{N(\fp)^\frac{1}{2}} \hat{\Phi}\left(\frac{\log N(\fp)}{\log R}\right)\sum_{\pi\in\cF(\tk,\fq)}\omega_{\pi}(x)\frac{L(1/2,\pi)^r}{L(1,\pi,\Ad)}\lambda_{\pi}(\fp), \\
    \cE_r^{(2)}(\tk,\fq;\Phi) &:= \frac{1}{\cM_r(\tk,\fq,\cO_F) \log R}\sum_{\fp \neq \fq} \frac{2\log N(\fp)}{N(\fp)}\hat{\Phi}\left(\frac{2\log N(\fp)}{\log R}\right)\sum_{\pi\in\cF(\tk,\fq)}\omega_{\pi}(x)\frac{L(1/2,\pi)^r}{L(1,\pi,\Ad)}\lambda_{\pi}(\fp^2).
\end{align*}
Here we used the fact that, for some small $\eta$,
\[\cM_r(\tk,\fq,\cO_F)=\sum_{\pi\in\cF(\tk,\fq)}\omega_{\pi}(x)\frac{L(1/2,\pi)^r}{L(1,\pi,\Ad)}+O((\|\tk\|N(\fq))^{1-\eta}).\]
Therefore, it suffices to study 
\[\sum_{\pi\in\cF(\tk,\fq)}\omega_{\pi}(x)\frac{L(1/2,\pi)^r}{L(1,\pi,\Ad)}\lambda_{\pi}(\fp^j)\]
for $j=1,2.$ By the definition of $\omega_{\pi}(x)$ and the Hecke relation, we have:
\begin{equation}\label{eq. main term, without, general case}
\sum_{\pi\in\cF(\tk,\fq)}\omega_{\pi}(x)\frac{L(1/2,\pi)^r}{L(1,\pi,\Ad)}\lambda_{\pi}(\fp^j)=\sum_{\substack{N(\fl\ff^2)\ll Q_r^{\varepsilon} \\ (\fl\ff,\fq)=1}}\frac{1}{N(\fl\ff^2)}\sum_{\fr|(\fp^j,\fl^2)}M_r(\mathbf{k},\fq,\fp^j\fl^2\fr^{-2}).
\end{equation}

\bigskip

\noindent\textbf{The case $r=2.$} Since $\hat{\Phi}$ is compactly supported in $(-\delta,\delta).$ We always have $N(\fp)\ll R^{\delta}= Q_2^{\delta}.$ By Corollary \ref{cor. second moment with harmonic weight}, we have:
\[M_2(\mathbf{k},\fq,\fp^2\fl^2\fr^{-2})=O\left(\frac{N(\fr)\log Q_2}{N(\fp)N(\fl)}\|\tk\|N(\fq)\right)+O(N(\fp\fl\fr^{-1})\|\tk\|^{1/2+\varepsilon}N(\fq)^{\varepsilon}).\]
This implies:
\[\sum_{\pi\in\cF(\tk,\fq)}\omega_{\pi}(x)\Omega_{\pi,2}\lambda_{\pi}(\fp^2)=O\left(\frac{\log Q_2}{N(\fp)}\|\tk\|N(\fq)\right)+O(N(\fp)\|\tk\|^{1/2+\varepsilon}N(\fq)^{\varepsilon}).\]
Therefore, (recall that $Q_2=R=\|\tk\|^2N(\fq)$)
\begin{equation}\label{eq. sec without p power}
	\cE_2^{(2)}(\tk,\fq;\Phi)=O\left(\frac{1}{\log \|\tk\|^2N(\fq)}\right)+O(\|\tk\|^{-1/2+\delta+\varepsilon}N(\fq)^{-1+\delta/2+\varepsilon}),
\end{equation} 
provided that $\hat{\Phi}$ is supported in $(-\delta,\delta).$

Next, apply Corollary \ref{cor. second moment with harmonic weight} again, and we obtain:
\[M_2(\mathbf{k},\fq,\fp\fl^2\fr^{-2})=C_2(\tk,\fq)\frac{2N(\fr)}{N(\fp)^{1/2}N(\fl)}(\log Q_2-\log N(\fp))+O(N(\fp)^{1/2}N(\fl\fr^{-1})\|\tk\|^{1/2+\varepsilon}N(\fq)^{\varepsilon}).\]
This yields:
\[\sum_{\pi\in\cF(\tk,\fq)}\omega_{\pi}(x)\Omega_{\pi,2}\lambda_{\pi}(\fp)=C_2(\tk,\fq)\sum_{\substack{N(\fl\ff^2)\ll Q_r^{\varepsilon} \\ (\fl\ff,\fq)=1}}\frac{1}{N(\fl^2\ff^2)}\sum_{\fr|(\fp,\fl^2)}\frac{2N(\fr)}{N(\fp)^{1/2}}(\log Q_2-\log N(\fp))+O(N(\fp)^{1/2}\|\tk\|^{1/2+\varepsilon}N(\fq)^{\varepsilon}).\]
The condition $\fr|(\fp,\fl^2)$ implies that $\fr=\cO_F$ or $\fr=\fp.$ In the second case, $\fr=\fp$ will force $\fp|\fl$, which will only contribute the error term. More precisely,
\begin{flalign*}
	\sum_{\pi\in\cF(\tk,\fq)}\omega_{\pi}(x)\Omega_{\pi,2}\lambda_{\pi}(\fp)&=\frac{2C_2(\tk,\fq)}{N(\fp)^{1/2}}(\log R-\log N(\fp))\sum_{\substack{N(\fl\ff^2)\ll Q_r^{\varepsilon} \\ (\fl\ff,\fq)=1}}\frac{1}{N(\fl^2\ff^2)}\\
	&\hspace{20mm}+O\left(\frac{\log R}{N(\fp)}\|\tk\|N(\fq)\right)+O(N(\fp)^{1/2}\|\tk\|^{1/2+\varepsilon}N(\fq)^{\varepsilon}).
\end{flalign*}
This yields, together with Lemma \ref{Landau prime ideal theorem with test function},
\begin{equation}\label{eq. sec moment without, p}
\cE_2^{(1)}(\tk,\fq;\Phi)= -4\int_{0}^\infty\hat{\Phi}(u)u\,du + 2\Phi(0)+O\left(\frac{1}{\log \|\tk\|^2N(\fq)}\right)+O(\|\tk\|^{-1/2+2\delta+\varepsilon}N(\fq)^{-1+\delta+\varepsilon}).
\end{equation}
Insert \eqref{eq. sec moment without, p} and \eqref{eq. sec without p power} into \eqref{explicit formula without harmonic weight, general}, and we conclude:
$$\lim_{\|\tk\|N(\fq) \to \infty}B_2(\tk,\fq;\Phi) = \hat{\Phi}(0) - \frac{3}{2}\Phi(0) + 4\int_0^\infty \hat{\Phi}(u)u\,du$$
as supp$(\hat{\Phi}) \subset(-\frac{1}{4},\frac{1}{4}).$ Then apply the Fourier transformations in \S \ref{subsec. fourier transformation} and we complete the proof.
\qed

\bigskip

\noindent\textbf{The case $r=3.$} In this case, we fix the weight $\tk$. Then $N(\fq)=N(\fq)\asymp R.$ 
Apply Lemma \ref{lem. useful lemma when removing the harmonic weight in llz}, and we obtain:
\begin{flalign*}
M_3(\tk,\fq,\fp^2\fl^2\fr^{-2})=O\left(\frac{\tau(\fl)^2N(\fr)\log^3 N(\fq)}{N(\fp)N(\fl)}\right)+O(N(\fq)^{1/2+\varepsilon}N(\fp\fl\fr^{-1})^{1+\varepsilon})
\end{flalign*}
Insert it into \eqref{eq. main term, without, general case} with $r=3$, which yields:
\[\sum_{\pi\in\cF(\tk,\fq)}\omega_{\pi}(x)\Omega_{\pi,3}\lambda_{\pi}(\fp^2)=O\left(\frac{\log^3 N(\fq)}{N(\fp)}N(\fq)\right)+O(N(\fp)N(\fq)^{\varepsilon}).\]
Therefore, (notice that $R\asymp N(\fq)=N(\fq)$)
\begin{equation}\label{eq. cubic without p power}
	\cE_3^{(2)}(\tk,\fq;\Phi)=O\left(\frac{1}{\log N(\fq)}\right)+O(N(\fq)^{-1/2+\delta/2+\varepsilon}),
\end{equation} 
provided that $\hat{\Phi}$ is supported in $(-\delta,\delta).$ 

Next, apply \eqref{eq. main term, without, general case} with $r=3$ and $j=1.$ Depending on $\fr=\cO_F$ or $\fr=\fp,$ we have:
\begin{equation}\label{eq. cubic case}
 \sum_{\pi\in\cF(\tk,\fq)}\omega_{\pi}(x)\Omega_{\pi,3}\lambda_{\pi}(\fp)=\sum_{\substack{N(\fl\ff^2)\ll N(\fq)^{\varepsilon} \\ (\fl\ff,\fq)=1}}\frac{1}{N(\fl\ff^2)}M_3(\mathbf{k},\fq,\fp\fl^2)+\sum_{\substack{N(\fp\fl\ff^2)\ll N(\fq)^{\varepsilon} \\ (\fl\ff,\fq)=1}}\frac{1}{N(\fp\fl\ff^2)}M_3(\mathbf{k},\fq,\fp\fl^2).  
\end{equation}
By Lemma \ref{lem. useful lemma when removing the harmonic weight in llz}, the second term will only contribute the error term:
\[\sum_{\substack{N(\fp\fl\ff^2)\ll N(\fq)^{\varepsilon} \\ (\fl\ff,\fq)=1}}\frac{1}{N(\fp\fl\ff^2)}M_3(\mathbf{k},\fq,\fp\fl^2)=O\left(\frac{\log^3 R}{N(\fp)}N(\fq)\right)+O(N(\fq)^{1/2+\varepsilon}N(\fp)^{-1/2+\varepsilon}).\]

Apply Lemma \ref{lem. useful lemma when removing the harmonic weight in llz}, Proposition \ref{cubic moment}, and we obtain:
\begin{flalign*}
\sum_{\substack{N(\fl\ff^2)\ll N(\fq)^{\varepsilon} \\ (\fl\ff,\fq)=1}}\frac{1}{N(\fl\ff^2)}M_3(\mathbf{k},\fq,\fp\fl^2)&=2C_2(\tk,\fq)\sum_{\substack{N(\fl\ff^2)\ll N(\fq)^{\varepsilon} \\ (\fl\ff,\fq)=1\\(\fp,\fl)=1}}\frac{M_3(\tk,\fq,\fp\fl^2)}{N(\fl\ff^2)}+O\left(\frac{\log^3R }{N(\fp)}N(\fq)\right)+O(N(\fq\fp)^{1/2+\varepsilon})\\\\
&=2C_2(\tk,\fq)\sum_{\substack{N(\fl\ff^2)\ll N(\fq)^{\varepsilon} \\ (\fl\ff,\fq)=1\\(\fp,\fl)=1}}\frac{1}{N(\fl\ff^2)}\{(\log N(\fq)+c_0(k))S_1(\fq,\fp\fl^2)-S_2(\fq,\fp\fl^2)\}\\
&\hspace{20mm}+O\left(\frac{\log^3R }{N(\fp)}N(\fq)\right)+O(N(\fq)^{1/2+\varepsilon}N(\fp)^{1/2+\varepsilon})\\
\end{flalign*}
Then we have, assuming that $(\fp,\fl)=1,$
\begin{align*}
    S_1(\fq,\fp\fl^2)&=\frac{\gamma_{-1}^2}{2 N(\fp)^{\frac{1}{2}}N(\fl)} \sum_{\fc|\fp\fl^2}\tau(\fp\fl^2\fc^{-1}) \left(\frac{\log N(\fq)}{2}-\log N(\fc)\right)^2\sum_{\fd|\fc} \frac{\mu(\fd)}{N(\fd)}\\
    &\hspace{10mm}+O\left(\frac{\tau(\fl^2)^2\log N(\fq)}{N(\fp)^{1/2}N(\fl)}\right)+O(N(\fq)^{-1/2+\varepsilon}N(\fp\fl^2)^{1/2+\varepsilon})\\
    &=\frac{\gamma_{-1}^2}{2 N(\fp)^{\frac{1}{2}}N(\fl)} \tau(\fp)\left(\frac{\log N(\fq)}{2}\right)^2\sum_{\fc|\fl^2}\tau(\fl^2\fc^{-1})\sum_{\fd|\fc} \frac{\mu(\fd)}{N(\fd)}\\
    &\hspace{10mm}+\frac{\gamma_{-1}^2}{2 N(\fp)^{\frac{1}{2}}N(\fl)}  \left(\frac{\log N(\fq)}{2}-\log N(\fp)\right)^2\sum_{\fc|\fl^2}\tau(\fl^2\fc^{-1})\sum_{\fd|\fc} \frac{\mu(\fd)}{N(\fd)}\\
    &\hspace{10mm}+O\left(\frac{\tau(\fl^2)^2\log^2N(\fl)\log N(\fq)}{N(\fp)^{1/2}N(\fl)}\right)+O(N(\fq)^{-1/2+\varepsilon}N(\fp\fl^2)^{1/2+\varepsilon})\\
\end{align*}
Here the first term comes from those $\fc$ only dividing $\fl^2$, and the second term comes form those $\fc$ divisible by $\fp.$ Notice that these two types $\fc$ do not overlap since $(\fp,\fl)=1.$

A similar argument will show:
\begin{flalign*}
    S_2(\fq,\fp\fl^2) &= \frac{\gamma_{-1}^2}{3N(\fp)^{1/2}N(\fl)} \left(\tau(\fp)\left(\frac{\log N(\fq)}{2}\right)^3+\left(\frac{\log N(\fq)}{2}-\log N(\fp)\right)^3 \right)\sum_{\fc|\fl^2}\tau(\fl^2\fc^{-1})\sum_{\fd|\fc} \frac{\mu(\fd)}{N(\fd)} \\
    &\hspace{10mm}+ \frac{\gamma_{-1}^2}{2N(\fp)^{1/2}N(\fl)} \tau(\fp)\log N(\fp)\left(\frac{\log^2N(\fq)}{2}\right)^2\sum_{\fc|\fl^2}  \tau(\fl^2\fc^{-1})\sum_{\fd|\fc} \frac{\mu(\fd)}{N(\fd)} \\
    &\hspace{10mm}+O\left(\frac{\tau(\fl^2)^2\log^3 N(\fl)\log^2 N(\fq)}{N(\fp)^{1/2}}\right)+O(N(\fq)^{-1/2+\varepsilon}N(\fp\fl^2)^{1/2+\varepsilon})
\end{flalign*}
Insert $S_1(\fq,\fp\fl^2)$ and $S_2(\fq,\fp\fl^2)$ into $M_3(\tk,\fq,\fp\fl^2)$ and then \eqref{eq. cubic case}, we obtain:
\begin{flalign*}
    \sum_{\pi\in\cF(\tk,\fq)}\omega_{\pi}(x)\Omega_{\pi,3}\lambda_{\pi}(\fp)&=\frac{2C_2(\tk,\fq)\gamma_{-1}^2}{N(\fp)^{1/2}}\cL(\log Q,\log N(\fq))\sum_{\substack{N(\fl\ff^2)\ll N(\fq)^{\varepsilon} \\ (\fl\ff,\fq)=1}}\frac{1}{N(\fl^2\ff^2)}\sum_{\fc|\fl^2}  \tau(\fl^2\fc^{-1})\sum_{\fd|\fc} \frac{\mu(\fd)}{N(\fd)} \\
    &\hspace{10mm}+O\left(\frac{\log^2N(\fq)}{N(\fp)^{1/2}}N(\fq)\right)+O(N(\fq)^{1/2+\varepsilon}N(\fp)^{1/2+\varepsilon}).
\end{flalign*}
Here $\cL(x,y)=\frac{x^3}{4}-\frac{x^2y}{2}+\frac{y^3}{3}.$
Then by \eqref{eq. cubic moment, without harmonic weight, fixed ideal}, we obtain:
\[\cM_3(\tk,\fq,\cO_F)=\frac{C_2(\tk,\fq)\gamma_{-1}^2}{6}\log^3 N(\fq)\sum_{\substack{N(\fl\ff^2)\ll N(\fq)^{\varepsilon} \\ (\fl\ff,\fq)=1}}\frac{1}{N(\fl^2\ff^2)}\sum_{\fc|\fl^2}  \tau(\fl^2\fc^{-1})\sum_{\fd|\fc} \frac{\mu(\fd)}{N(\fd)}+O(N(\fq)\log^2 N(\fq)).\]
Notice that $R\asymp N(\fq)=N(\fq)$. Together with Lemma \ref{Landau prime ideal theorem with test function},
\begin{equation}\label{eq. cubic mo, without, prime}
\cE_3^{(1)}(\tk,\fq;\Phi)=3\Phi(0)-12\int_0^\infty\ \hat{\Phi}(u)u\,du +8\int_0^\infty \hat{\Phi}(u)u^3\,du+O\left(\frac{1}{\log N(\fq)}\right)+O(N(\fq)^{-1/2+\delta+\varepsilon})
\end{equation}
Insert \eqref{eq. cubic mo, without, prime} and \eqref{eq. cubic without p power} into \eqref{explicit formula without harmonic weight, general}, and we conclude:
$$\lim_{N(\fq) \to \infty}B_3(\tk,\fq;\Phi) = \hat{\Phi}(0)-\frac{5}{2}\Phi(0)+12\int_0^\infty\ \hat{\Phi}(u)u\,du -8\int_0^\infty \hat{\Phi}(u)u^3\,du$$
as supp$(\hat{\Phi}) \subset (-\frac{1}{2},\frac{1}{2}).$ Then apply the Fourier transformations in \S \ref{subsec. fourier transformation} and we complete the proof.
\qed

\section{General Conjectures}\label{sec. main conjectures and proofs}

For simplicity, we only consider the case $F=\bQ$, $q=1$ and $k\equiv0\Mod{4},$ as the proofs are identical. Let $r\geq1$ be an integer. Recall that, in \S \ref{subsec, cheby poly}, we introduced the Chebyshev polynomial $T_{\ell}(x)$ for $\ell$ being an integer.

\subsection{The Weighted Moment Conjecture}\label{apx. weighted moment conjecture}
In this section, we will propose the weighted moment conjecture, which is a generalization of the moment conjecture of $L$-functions. (e.g., \cite{ConreyFarmerKeatingRubinsteinSnaith2005}, \cite{DiaconuGoldfeldHoffstein2003}.)  

Let $f\in H_k(1)$ be a Hecke eigenform. We assume that $f$ has a Fourier expansion:
\[f(z)=\sum_{m=1}^{\infty}\lambda_f(m)m^{\frac{k-1}{2}}e(mz)\hspace{10mm}\mbox{with $\lambda_f(1)=1$}.\]
Then for $m_1,\ldots,m_r\geq1$ be integers. We introduce the notation:
\begin{equation}\label{eq. delta function}
\delta(m_1,\ldots,m_r)=\lim_{k\to\infty}\sum_{f\in H_k(1)}\frac{2\pi^2 \lambda_f(m_1)\cdots\lambda_f(m_r)}{(k-1)L(1,\sym^2f)}.
\end{equation}
We can show:
\begin{prop}\label{prop. basic properties for delta}
	$\delta(m_1,\ldots,m_r)$ has the following properties:
	\begin{enumerate}
		\item $\delta(1,m_1,\ldots,m_r)=\delta(m_1,\ldots,m_r).$
		\item (multiplicativity) Suppose that $(m_1\cdots m_r,m_1'\cdots m_r')=1,$ then 
\[\delta(m_1m_1',\ldots,m_rm_r')=\delta(m_1,\ldots,m_r)\delta(m_1',\ldots,m_r').\]
		\item Let $\ell_1,\dots,\ell_r\geq0$ be integers, then
\[\delta(p^{\ell_1},\ldots,p^{\ell_r})=\frac{2}{\pi}\int_0^{\pi}\sin^2\theta\prod_{j=1}^r\frac{e^{i(\ell_j+1)\theta}-e^{-i(\ell_j+1)\theta}}{e^{i\theta}-e^{-i\theta}}\,d\theta=\int_{\bR}T_{\ell_1}(x)\cdots T_{\ell_r}(x)\,d\mu_{\infty}(x).\]
	\end{enumerate}
\end{prop}
\begin{proof}
	(1) and (2) can be checked by using the definition. (3) is Lemma 3.1.2 in \cite{ConreyFarmerKeatingRubinsteinSnaith2005}.
\end{proof}

Next, for $\ell\geq1$ and $(z_1,\ldots,z_r)\in\bC^r$ in a small enough neighborhood of $(0,\ldots,0)$. We define the function
\begin{flalign*}
	A(\ell;z_1,\ldots,z_r)&=\prod_{p}A_p(\ord_p(\ell);z_1,\ldots,z_r)
\end{flalign*}
with
\[A_p(m;z_1,\ldots,z_r)=\prod_{1\leq i<j\leq r}\left(1-\frac{1}{p^{1+z_i+z_j}}\right)\frac{2}{\pi}\int_0^{\pi}\sin^2\theta\frac{\sin((m+1)\theta)}{\sin\theta}\prod_{j=1}^r\frac{1}{\left(1-\frac{2\cos\theta}{p^{1/2+z_j}}+\frac{1}{p^{1+2z_j}}\right)}\,d\theta.\]
This is absolutely convergent when $|\Re(z_j)|<\frac{1}{2}.$ Compare $A_p(0;0,\ldots,0)$ with \eqref{eq. local component}:
\begin{equation}\label{eq. local Ap}
	A_p(0;0,\ldots,0)=\left(1-\frac{1}{p}\right)^{\frac{r(r-1)}{2}}a(p,r).
\end{equation}

Utilizing the method in \cite{ConreyFarmerKeatingRubinsteinSnaith2005}, we formula the following conjecture for the eignevalue-weighted moment conjecture:
\begin{conj}\label{conj.main weighted moment conjecture}
	Let $\ell\geq1$ and $r\geq1$ be integers. Set $X_k(s)=\frac{\Gamma\left(\frac{1}{2}-s+\frac{k}{2}\right)}{\Gamma\left(s-\frac{1}{2}+\frac{k}{2}\right)}.$ Then
	\begin{equation}\label{eq.main weighted moment conjecture}
	\sum_{f\in H_k(1)}\frac{2\pi^2 \lambda_f(\ell)L\left(\frac{1}{2},f\right)^r}{(k-1)L(1,\sym^2f)}=R_{k,r}(\ell)(1+O((\ell k)^{-\delta+\varepsilon}))
	\end{equation}
	for some positive $\delta>0.$ Here $R_{k,r}(\ell)$ is given by the $r$-fold residue:
	\[R_{k,r}(\ell)=\frac{(-1)^{r(r-1)/2}2^{r}}{r!}\frac{1}{(2\pi i)^r}\oint\cdots\oint\frac{G(\ell;z_1,\ldots,z_r)\Delta(z_1^2,\ldots,z_r^2)^2}{z_1^{2r-1}\cdots z_r^{2r-1}}\prod_{j=1}^{r}X_k\left(\frac{1}{2}+z_j\right)^{-1/2}\,dz_1\cdots\,dz_r,\]
	where
\[G(\ell;z_1,\ldots,z_r)=A(\ell;z_1,\ldots,z_r)\prod_{1\leq i<j\leq r}\zeta(1+z_i+z_j).\]
\end{conj} 
We will call $R_{k,r}(\ell)$ the \textit{main term} for the weighted moment conjecture.

\begin{remark}\label{rem.possible changes}
	When the level $q\neq1$ is prime, we need to replace $X_k(s)$ by $X_{k,q}(s)=\left(\frac{q}{4\pi^2}\right)^{1/2-s}\frac{\Gamma\left(\frac{1}{2}-s+\frac{k}{2}\right)}{\Gamma\left(s-\frac{1}{2}+\frac{k}{2}\right)}.$ Additionally, we need to replace $2^r$ (in $R_{k,r}(\ell)$) by $2^{r-1}$ due to the expectation for root numbers. 
\end{remark}

\noindent\textbf{Idea of the conjecture}: We only sketch how to propose the conjecture, as it is identical to that of \cite[Conjecture 1.5.5]{ConreyFarmerKeatingRubinsteinSnaith2005}. 
	
	Following the argument in \cite[Section 4.5]{ConreyFarmerKeatingRubinsteinSnaith2005}, we conjecture: for $(\alpha_1,\ldots,\alpha_r)\in\bC^r$ in a small neighborhood of $(0,\ldots,0),$
\begin{flalign*}
&\sum_{f\in H_k(1)}\frac{2\pi^2 \lambda_f(\ell)L\left(\frac{1}{2}+\alpha_1,f\right)\cdots L\left(\frac{1}{2}+\alpha_r,f\right)}{(k-1)L(1,\sym^2f)}\\
&\hspace{20mm}=\prod_{j=1}^rX_k\left(\frac{1}{2}-\alpha_j\right)^{-1/2}\sum_{j=1}^r\sum_{\varepsilon_j\in\{\pm1\}}	K(\ell;\varepsilon_1\alpha_1,\ldots,\varepsilon_r\alpha_r)(1+O((\ell k)^{-\delta+\varepsilon})).	
\end{flalign*}
where
\[K(\ell;z_1,\ldots,z_r)=A(\ell;z_1,\ldots,z_r)\prod_{j=1}^rX_k\left(\frac{1}{2}+z_j\right)^{-1/2}\prod_{1\leq i<j\leq r}\zeta(1+z_i+z_j).\]
(Notice that, when the level $q>1,$ we need to further require that $\prod\limits_{j=1}^r\varepsilon_j=1.$ This explains why we replace $2^r$ by $2^{r-1}$ in Remark \ref{rem.possible changes}.)

Obviously, $K(\ell;z_1,\ldots,z_r)$ satisfies the conditions in \cite[Lemma 2.5.2]{ConreyFarmerKeatingRubinsteinSnaith2005}. Apply  \cite[Lemma 2.5.2]{ConreyFarmerKeatingRubinsteinSnaith2005}, take the limit $(\alpha_1,\ldots,\alpha_r)\to(0,\ldots,0)$, and we complete the proof.
\qed

\bigskip
\bigskip

Next, we will refine the main term $R_{k,r}(\ell)$ in the weighted moment conjecture. Our main tool is the Laurent expansion of $A(\ell;z_1,\ldots,z_r)$. The following lemma shows that, all but finite terms in the Laurent expansion will vanish:
\begin{lemma}\label{lem. most terms vanish}
	Let $r\geq1$ be a fixed integer. Then for any integers $\alpha=\alpha_1+\cdots\alpha_r\geq\frac{r(r-1)}{2}+1,$ and holomorphic function $f$ (around $(0,\ldots,0)$),  
	\[\frac{1}{(2\pi i)^n}\oint\cdots\oint\frac{\Delta(z_1^2,\ldots,z_r^2)\Delta(z_1,\ldots,z_r)z_1^{\alpha_1}\cdots z_r^{\alpha_r}}{z_1^{2r-1}\cdots z_r^{2r-1}}f(z_1,\ldots,z_r)\,dz_1\cdots\,dz_r=0.\]
\end{lemma}
\begin{proof}
	Denote by $S_r$ the permutation group of $\{0,\ldots,r-1\}$. Then by the permutation expansion,
	\[\Delta(z_1,\ldots,z_r)=\sum_{\sigma\in S_r}\sgn(\sigma)z_1^{\sigma(0)}\cdots z_r^{\sigma(r-1)}.\]
	This will imply:
	\[\frac{\Delta(z_1^2,\ldots,z_r^2)\Delta(z_1,\ldots,z_r)z_1^{\alpha_1}\cdots z_r^{\alpha_r}}{z_1^{2r-1}\cdots z_r^{2r-1}}=\sum_{\sigma,\sigma'}\sgn(\sigma\sigma')\frac{z_1^{2\sigma(0)+\sigma'(0)+\alpha_1}\cdots z_r^{2\sigma(r-1)+\sigma'(r-1)+\alpha_r}}{z_1^{2r-1}\cdots z_r^{2r-1}}.\]
	Fix $\sigma,\sigma'\in S_r$. Then we claim that, if $\alpha=\alpha_1+\cdots\alpha_r\geq\frac{r(r-1)}{2}+1,$ then $2\sigma(j)+\sigma'(j)+\alpha_j\geq 2r-1$ for some $1\leq j\leq r.$ Since we chose $\sigma,\sigma'$ arbitrarily, Cauchy's integral theorem will imply that the integral vanishes.
	
	We prove the claim by contradiction. Suppose not. Then $2\sigma(j)+\sigma'(j)+\alpha_j\leq 2r-2$ for all $1\leq j\leq r.$ Then summing over all $j,$
	\[\sum_{j=1}^r2\sigma(j)+\sigma'(j)+\alpha_j=\frac{3r(r-1)}{2}+\sum_{j=1}^r\alpha_j\leq r(2r-2).\]
	This will force $\alpha_1+\cdots\alpha_r\leq \frac{r(r-1)}{2}.$ A contradiction.
\end{proof}

Conjecture \ref{conj.main weighted moment conjecture} yields the following corollary:

\begin{cor}\label{cor. weighted moment, general}
	Assume Conjecture \ref{conj.main weighted moment conjecture}. Let $\ell\geq1$ be an integer and $r\geq1$. Then
	\[\sum_{f\in H_k(1)}\frac{2\pi^2 \lambda_f(\ell)L\left(\frac{1}{2},f\right)^r}{(k-1)L(1,\sym^2f)}=(\log k)^{\frac{r(r-1)}{2}}\left(\prod_{p\nmid\ell}A_p(0;0,\ldots,0)\right)g_r(\ell)\left(1+O\left(\frac{1}{\log k}\right)+O((\ell k)^{-\delta+\varepsilon})\right),\]
where
\[g_r(\ell)=\frac{(-1)^{r(r-1)/2}2^{r}}{r!(2\pi i)^r}\oint\cdots\oint\frac{\Delta(z_1^2,\ldots,z_r^2)\Delta(z_1,\ldots,z_r)\prod\limits_{p|\ell}A_p\left(\ord_p(\ell);\frac{z_1}{\log k},\ldots,\frac{z_r}{\log k}\right)}{z_1^{2r-1}\cdots z_r^{2r-1}}e^{z_1+\cdots+z_r}\,dz_1\cdots\,dz_r,\]
\end{cor}

\begin{proof}
	Here we only outline the proof, as it is an analogue of \cite[Section 2.7]{ConreyFarmerKeatingRubinsteinSnaith2005}.
	
	By the Stirling formula, we can show:
	\[\prod_{j=1}^rX_k\left(\frac{1}{2}+z_j\right)^{-1/2}=e^{(\log k)(z_1+\cdots+z_r)}\left(1+O\left(\frac{1}{\log k}\right)\right).\]
	
	Insert it into \eqref{eq.main weighted moment conjecture}, do the change of variable $z_i\mapsto \frac{z_i}{\log k}$ and apply the Laurent expansion 
	\[\zeta\left(1+\frac{z_i+z_j}{\log k}\right)=\frac{\log k}{z_i+z_j}+O(1).\] 
	This will imply: the main term $R_{k,r}(\ell)$ in the weighted moment 
	\begin{flalign*}
	R_{k,r}(\ell)=&\frac{(-1)^{r(r-1)/2}2^{r}}{r!}(\log k)^{\frac{r(r-1)}{2}}\left(1+O\left(\frac{1}{\log k}\right)\right)\\
	    &\times\frac{1}{(2\pi i)^r}\oint\cdots\oint\frac{\Delta(z_1^2,\ldots,z_r^2)\Delta(z_1,\ldots,z_r)A\left(\ell;\frac{z_1}{\log k},\ldots,\frac{z_r}{\log k}\right)}{z_1^{2r-1}\cdots z_r^{2r-1}}e^{z_1+\cdots+z_r}\,dz_1\cdots\,dz_r\\
	\end{flalign*}	
	
	Next, we write
	\[A\left(\ell;\frac{z_1}{\log k},\ldots,\frac{z_r}{\log k}\right)=\prod_{p\nmid \ell}A_p\left(0;\frac{z_1}{\log k},\ldots,\frac{z_r}{\log k}\right)\prod_{p|\ell}A_p\left(\ord_p(\ell);\frac{z_1}{\log k},\ldots,\frac{z_r}{\log k}\right).\]
	
	Similar to the argument in \cite[Theorem 2.4.1]{ConreyFarmerKeatingRubinsteinSnaith2005}, we can find a product expression for the first part. More precisely, for each fixed $p$, we can show:
	\[\prod_{1\leq i<j\leq r}\left(1-\frac{1}{p^{1+z_i+z_j}}\right)=1-\sum_{1\leq i<j\leq r}\frac{1}{p^{1+z_i+z_j}}+O\left(\frac{1}{p^{7/4}}\right),\]
	when $z_i$ are small enough. On the other hand,
\begin{flalign*}
	\frac{2}{\pi}\int_0^{\pi}\sin^2\theta\prod_{j=1}^r\frac{1}{\left(1-\frac{2\cos\theta}{p^{1/2+z_j}}+\frac{1}{p^{1+2z_j}}\right)}\,d\theta&=\sum_{\ell_1,\ldots,\ell_r}p^{-\ell_1(1/2+z_1)-\cdots-\ell_r(1/2+z_r)}\int_{\bR}T_{\ell_1}(x)\cdots T_{\ell_r}(x)\,d\mu_{\infty}(x)\\
	&=1+\sum_{1\leq i<j\leq r}\frac{1}{p^{1+z_i+z_j}}+O\left(\frac{1}{p^{5/4}}\right)
\end{flalign*}
	This will imply: $A_p\left(0;0,\ldots,0\right)=1+O\left(\frac{1}{p^{5/4}}\right)$. Moreover, for $1\leq\alpha=\alpha_1+\cdots+\alpha_r\leq\frac{r(r-1)}{2},$ we have
	\[\frac{\partial^{\alpha}A_p}{\partial z_1^{\alpha_1}\cdots\partial z_r^{\alpha_r}}(0;0,\ldots,0)=O\left(\frac{1}{p^{5/4}}\right).\]
Therefore, apply Lemma \ref{lem. most terms vanish} and utilize the Laurent expansion for $A\left(\ell;\frac{z_1}{\log k},\ldots,\frac{z_r}{\log k}\right)$, which yields: 
	\begin{flalign*}
	R_{k,r}(\ell)=&\frac{(-1)^{r(r-1)/2}2^{r}}{r!}(\log k)^{\frac{r(r-1)}{2}}\left(\prod_{p\nmid\ell}A_p(0;0,\ldots,0)\right)\left(1+O\left(\frac{1}{\log k}\right)\right)\\
	    &\times\frac{1}{(2\pi i)^r}\oint\cdots\oint\frac{\Delta(z_1^2,\ldots,z_r^2)\Delta(z_1,\ldots,z_r)\prod\limits_{p|\ell}A_p\left(\ord_p(\ell);\frac{z_1}{\log k},\ldots,\frac{z_r}{\log k}\right)}{z_1^{2r-1}\cdots z_r^{2r-1}}e^{z_1+\cdots+z_r}\,dz_1\cdots\,dz_r\\
	\end{flalign*}
	This completes the proof.
\end{proof}

\subsection{The Weighted Low-lying Zeros Conjecture}

\noindent\textbf{Proof of Conjecture \ref{conj. arithmetic low lying zero}}:
Assume that $\hat{\phi}$ is a smooth function compactly supported in $(-\eta,\eta)$ with $\eta$ being small. We start with:
\begin{flalign*}
M_r(k,p^{\ell}):=\frac{1}{\sum_{f\in H_k(1)}\frac{L(1/2,f)^r}{L(1,\sym^2f)}}\sum_{f\in H_k(1)}\frac{\lambda_f(p^{\ell})L(1/2,f)^r}{L(1,\sym^2f)}&=\frac{g_r(p^{\ell})}{A_p(0;0,\ldots,0)g_r(1)}+O\left(\frac{1}{\log k}\right)\\
	&=\frac{g_r(p^{\ell})}{g_r(1)}+O\left(\frac{1}{p^{5/4}}\right)+O\left(\frac{1}{\log k}\right)
\end{flalign*}
for $\ell=1,2$ and prime $p\ll k^{\eta}.$ The first equality is due to Corollary \ref{cor. weighted moment, general}, and the second equality is due to $A_p(0;0,\ldots,0)=1+O\left(\frac{1}{p^{5/4}}\right),$ which was shown in Corollary \ref{cor. weighted moment, general}.

We first study $M_r(k,p^2),$ that is, $\ell=2.$ By Corollary \ref{cor. weighted moment, general}, it suffices to study:
\[g_r(p^{2})=\frac{(-1)^{r(r-1)/2}2^{r}}{r!(2\pi i)^r}\oint\cdots\oint\frac{\Delta(z_1^2,\ldots,z_r^2)\Delta(z_1,\ldots,z_r)A_p\left(2;\frac{z_1}{\log k},\ldots,\frac{z_r}{\log k}\right)}{z_1^{2r-1}\cdots z_r^{2r-1}}e^{z_1+\cdots+z_r}\,dz_1\cdots\,dz_r.\]
Recall that
\[A_p(2;z_1,\ldots z_r)=\prod_{1\leq i<j\leq r}\left(1-\frac{1}{p^{1+z_1+z_j}}\right)\frac{2}{\pi}\int_0^{\pi}\sin^2\theta\frac{\sin(3\theta)}{\sin\theta}\prod_{j=1}^r\frac{1}{\left(1-\frac{2\cos\theta}{p^{1/2+z_j}}+\frac{1}{p^{1+2z_j}}\right)}\,d\theta.\]
Again, by change of variable $x=2\cos\theta$ and open the parentheses, we can show:  

(a) $A_p\left(2;0,\ldots,0\right)=O\left(\frac{1}{p^{3/4}}\right)$ when $p$ is sufficiently large. 

(b) For $1\leq\alpha=\alpha_1+\cdots+\alpha_r\leq\frac{r(r-1)}{2},$ we have
	\[\frac{\partial^{\alpha}A_p}{\partial z_1^{\alpha_1}\cdots\partial z_r^{\alpha_r}}(2;0,\ldots,0)=O\left(\frac{1}{p^{3/4}}\right).\]

This implies:
\[M_r(k,p^2)=O\left(\frac{1}{p^{3/4}}\right)\]
and hence
\begin{equation}\label{eq. weighted by prime square, llz}
	\frac{1}{\sum_{f\in H_k(1)}\frac{L(1/2,f)^r}{L(1,\sym^2f)}}\sum_{f\in H_k(1)}\frac{L(1/2,f)^r}{L(1,\sym^2f)}\sum_{p}\frac{\lambda_f(p^2)\log p}{p}\hat{\phi}\left(\frac{\log p}{\log k}\right)=O(1).
\end{equation}

Next, we study $M_r(k,p),$ that is, $\ell=1.$ Set:
\[B_p(1;z_1,\ldots,z_r)=A_p(1;z_1,\ldots,z_r)-\sum_{i=1}^rp^{-(1/2+z_i)}=A_p(1;z_1,\ldots,z_r)-M_p(z_1,\ldots,z_r).\]
A similar argument will show:

(a) $B_p\left(1;0,\ldots,0\right)=O\left(\frac{1}{p^{3/4}}\right)$ when $p$ is sufficiently large. 

(b) For $1\leq\alpha=\alpha_1+\cdots+\alpha_r\leq\frac{r(r-1)}{2},$ we have
	\[\frac{\partial^{\alpha}B_p}{\partial z_1^{\alpha_1}\cdots\partial z_r^{\alpha_r}}(1;0,\ldots,0)=O\left(\frac{1}{p^{3/4}}\right).\]
	
This implies: 
\begin{flalign*}
g_r(p)&=\frac{(-1)^{r(r-1)/2}2^{r}}{r!(2\pi i)^r}\oint\cdots\oint\frac{\Delta(z_1^2,\ldots,z_r^2)\Delta(z_1,\ldots,z_r)M_p\left(\frac{z_1}{\log k},\cdots,\frac{z_r}{\log k}\right)}{z_1^{2r-1}\cdots z_r^{2r-1}}e^{z_1+\cdots+z_r}\,dz_1\cdots\,dz_r+O\left(\frac{1}{p^{3/4}}\right)\\
&=\frac{1}{p^{1/2}}\frac{(-1)^{r(r-1)/2}2^{r}}{(r-1)!(2\pi i)^r}\oint\cdots\oint\frac{\Delta(z_1^2,\ldots,z_r^2)\Delta(z_1,\ldots,z_r)}{z_1^{2r-1}\cdots z_r^{2r-1}}p^{-\frac{z_1}{\log k}}e^{z_1+\cdots+z_r}\,dz_1\cdots\,dz_n+O\left(\frac{1}{p^{3/4}}\right).
\end{flalign*}
The last equality is due to the symmetry of the integrand in $(z_1,\ldots,z_r).$

Then apply the Laurent expansion for $p^{-z}$, and we obtain:
\[g_r(p)=\frac{1}{p^{1/2}}\sum_{j\geq 0}\frac{(-1)^j}{j!}\left(\frac{\log p}{\log k}\right)^jb_{r}(j)+O\left(\frac{1}{p^{3/4}}\right)\]
with
\[b_{r}(j)=\frac{(-1)^{n(n-1)/2}2^{r}}{r!(2\pi i)^r}\oint\cdots\oint\frac{\Delta(z_1^2,\ldots,z_r^2)\Delta(z_1,\ldots,z_r)(z_1^{j}+\cdots z_r^j)}{z_1^{2r-1}\cdots z_r^{2r-1}}e^{z_1+\cdots+z_r}\,dz_1\cdots\,dz_r.\]
Notice that this is a finite sum and the sum is up to $2r-2.$ We also have: $b_r(0)=rg_r(1)$
\begin{remark}\label{rem, non-vanishing bnj}
	Combine with Lemma \ref{lem. most terms vanish}, $b_{r}(j)=0$ unless $j\leq\min\{2r-2,\frac{r(r-1)}{2}\}$
\end{remark}
Insert the new expression for $b_r(j)$ into $M_r(k,p):$
\[M_r(k,p)=\frac{1}{p^{1/2}}\sum_{j\geq 0}\frac{(-1)^jb_{r}(j)}{j!g_{r}(1)}\left(\frac{\log p}{\log k}\right)^{j}.\]
Sum over $p$ and apply Lemma \ref{Landau prime ideal theorem with test function}, we obtain: 
\begin{flalign*}
&\frac{1}{\log k}\frac{1}{\sum_{f\in H_k(1)}\frac{L(1/2,f)^r}{L(1,\sym^2f)}}\sum_{f\in H_k(1)}\frac{L(1/2,f)^r}{L(1,\sym^2f)}\sum_{p}\frac{ \lambda_f(p)\log p}{p^{1/2}}\hat{\phi}\left(\frac{\log p}{\log k}\right)\\
&\hspace{60mm}=\sum_{j\geq 0}\frac{(-1)^jb_{r}(j)}{j!g_{r}(1)}\sum_{p}\frac{1}{p}\left(\frac{\log p}{\log k}\right)^{j+1}\hat{\phi}\left(\frac{\log p}{2\log k}\right)+O\left(\frac{1}{\log k}\right)\\
&\hspace{60mm}=\int_{0}^{\infty}\hat{\phi}(u)\sum_{j\geq 0}\frac{(-1)^j2^{j+1}b_{r}(j)u^j}{j!g_{r}(1)}\,du+O\left(\frac{1}{\log k}\right).
\end{flalign*}
Next, assume that $\hat{\phi}$ is compactly supported in $(-\eta,\eta)\subseteq(-1,1).$ Then according to \S \ref{subsec. fourier transformation}, 
\[\int_0^{\infty}\hat{\phi}(u)u^n\,du=\frac{1}{2}\int_{-\infty}^{\infty}\widehat{\phi}(u)|u|^n\,du=\int_{-\infty}^{\infty}\phi(x)h_n(x)\,dx\]
where $h_n(x)$ is defined in \eqref{eq. the h function}.

This implies:
\begin{equation}\label{eq. weighted by prime, llz}
\begin{split}
& \frac{1}{\log k}\frac{1}{\sum_{f\in H_k(1)}\frac{L(1/2,f)^r}{L(1,\sym^2f)}}\sum_{f\in H_k(1)}\frac{L(1/2,f)^r}{L(1,\sym^2f)}\sum_{p}\frac{ \lambda_f(p)\log p}{p^{1/2}}\hat{\phi}\left(\frac{\log p}{\log k}\right) \\
&\hspace{60mm}= \int_{-\infty}^{\infty}\phi(x)\sum_{j\geq 0}\frac{(-1)^j2^{j+1}b_{r}(j)h_j(x)}{j!g_{r}(1)}\,dx+O\left(\frac{1}{\log k}\right).
\end{split}
\end{equation}
Recall the explicit formula:
\begin{equation}\label{eq. simple explicit formula}
    D(f,\phi)= \hat{\phi}(0) + \frac{1}{2}\phi(0) - \frac{1}{\log k}\sum_{p} \frac{\lambda_f(p)\log p}{p^\frac{1}{2}}\hat{\phi}\left(\frac{\log p}{2\log k}\right) - \frac{1}{\log k}\sum_{p} \frac{\lambda_f(p^2)\log p}{p}\hat{\phi}\left(\frac{\log p}{\log k}\right) 
\end{equation}
Combine \eqref{eq. simple explicit formula}, \eqref{eq. weighted by prime, llz} and \eqref{eq. weighted by prime square, llz}, which yields:
\[\lim_{k\to\infty}\frac{1}{\sum_{f\in H_k(1)}\frac{L(1/2,f)^r}{L(1,\sym^2f)}}\sum_{f\in H_k(1)}\frac{L(1/2,f)^r}{L(1,\sym^2f)}D(f,\phi)=\int_{-\infty}^{\infty}\phi(x)W_r(x)\,dx,\]
where
\[W_r(x)=1+h_0(x)-r\sum_{j\geq 0}\frac{(-1)^j2^{j+1}b_{r}(j)}{j!b_{r}(0)}h_j(x).\]
\qed

\subsection{Numerical Results for $b_r(j):$ $r=2,3,4$}\label{apx. explicit calculation for bnj}

In this section, we calculate $b_r(j)$ explicitly for $r=2,3,4.$

When $r=2,$ it suffices to calculate $b_2(0)$ and $b_2(1)$ by Remark \ref{rem, non-vanishing bnj}. A direct calculation will show:
\[b_2(0)=\frac{-4}{2!}(-4)\hspace{10mm} \mbox{and} \hspace{10mm} b_2(1)=\frac{-4}{2!}(-2)\]
Insert it into $W_2(x)$ and we obtain the result for $r=2$ in Theorem \ref{thm. llz, with harmonic weight}.

When $r=3,$ it suffices to calculate $b_3(j)$, $0\leq j\leq 3$, by Remark \ref{rem, non-vanishing bnj}. A direct calculation will show:
\[b_3(0)=\frac{-8}{3!}(-6),\hspace{5mm} b_3(1)=\frac{-8}{3!}(-6), \hspace{5mm} b_2(1)=0,\hspace{5mm} \mbox{and} \hspace{5mm} b_3(3)=\frac{-8}{3!}(6)\]
Insert it into $W_3(x)$ and we obtain the result for $r=3$ in Theorem \ref{thm. llz, with harmonic weight}.

When $r=4,$ it suffices to calculate $b_4(j)$, $j=0,1,3,5$. By using Mathematica, we can show:
\[b_4(0)=\frac{16}{4!}\frac{32}{15},\hspace{5mm} b_4(1)=\frac{16}{4!}\frac{16}{5}, \hspace{5mm} b_4(3)=\frac{16}{4!}(-8),\hspace{5mm} \mbox{and} \hspace{5mm} b_4(5)=\frac{16}{4!}(24).\]
This finishes the claim in Remark \ref{rem. comparison between W and W(SO)}.

\subsection{The Weighted Equidistribution Conjecture of $\lambda_\pi(\fp)$}\label{apx. arithmetic Sato-Tate}

Assuming Conjecture \ref{conj.main weighted moment conjecture}, we prove Conjecture \ref{conj. arithmetic Sato-Tate}.

\noindent\textbf{Proof of Conjecture \ref{conj. arithmetic Sato-Tate}}:
Let $p$ be a fixed prime and $\ell\geq0$ a fixed integer. Then, by Corollary \ref{cor. weighted moment, general},
\begin{flalign*}
	R_{k,r}(p^{\ell})=&\frac{(-1)^{r(r-1)/2}2^{r}}{r!}(\log k)^{\frac{r(r-1)}{2}}A(p^{\ell};0,\ldots,0)\left(1+O\left(\frac{1}{\log k}\right)\right)\\
	    &\times\frac{1}{(2\pi i)^r}\oint\cdots\oint\frac{\Delta(z_1^2,\ldots,z_r^2)\Delta(z_1,\ldots,z_r)}{z_1^{2r-1}\cdots z_r^{2r-1}}e^{z_1+\cdots+z_r}\,dz_1\cdots\,dz_r\\
	\end{flalign*}

Recall \eqref{eq. local Ap} and the definition of $A(p^{\ell};0\ldots,0)$, and we obtain:
\begin{flalign*}
\lim_{k\to\infty}\frac{1}{\sum_{f\in H_k(1)}\frac{L(1/2,f)^r}{L(1,\sym^2f)}}\sum_{f\in H_k(1)}\frac{\lambda_f(p^{\ell})L(1/2,f)^r}{L(1,\sym^2f)}&=\frac{A(p^{\ell};0\ldots,0)}{A(1;0\ldots,0)}\\
&=\frac{1}{a(p,r)}\frac{2}{\pi}\int_0^{\pi}\sin^2\theta\frac{\sin(\ell+1)\theta}{\sin\theta}\frac{1}{\left(1-\frac{2\cos\theta}{p^{1/2}}+\frac{1}{p}\right)^r}\,d\theta\\
&=\frac{1}{a(p,r)}\int_{\bR}T_{\ell}(y)\frac{1}{\left(1-\frac{y}{p^{1/2}}+\frac{1}{p}\right)^r}\,d\mu_{\infty}(y).
\end{flalign*}
(The last equality is due to the change of variable $y=2\cos\theta.$) Recall in \S \ref{subsec, cheby poly}, we showed:
\[d\mu_{p,r}(x)=\sum_{\ell\geq0}F_{p,r,\ell}T_{\ell}(x)\,d\mu_{\infty}(x),\]
with:
\[F_{p,r,\ell}=\lim_{k\to\infty}\frac{1}{\sum_{f\in H_k(1)}\frac{L(1/2,f)^r}{L(1,\sym^2f)}}\sum_{f\in H_k(1)}\frac{\lambda_f(p^{\ell})L(1/2,f)^r}{L(1,\sym^2f)}.\]
Therefore,
\[d\mu_{p,r}(x)=\frac{1}{a(p,r)}\int_{\bR}\left(\sum_{\ell\geq0}T_{\ell}(y)T_{\ell}(x)\right)\frac{1}{\left(1-\frac{y}{p^{1/2}}+\frac{1}{p}\right)^r}\,d\mu_{\infty}(y)\,d\mu_{\infty}(x).\]
Notice that $\{T_{\ell}(x)\}_{\ell\geq0}$ is an orthonormal basis for the space $C([-2,2])$ with respect to $\,d\mu_{\infty}$, and hence $\sum_{\ell\geq0}T_{\ell}(y)T_{\ell}(x)$ is reproducing kernel. This implies:
\[\int_{\bR}\left(\sum_{\ell\geq0}T_{\ell}(y)T_{\ell}(x)\right)\frac{1}{\left(1-\frac{y}{p^{1/2}}+\frac{1}{p}\right)^r}\,d\mu_{\infty}(y)=\frac{1}{\left(1-\frac{x}{p^{1/2}}+\frac{1}{p}\right)^r},\]
which concludes:
\[d\mu_{p,r}(x)=\frac{1}{a(p,r)}\frac{1}{\left(1-\frac{x}{p^{1/2}}+\frac{1}{p}\right)^r}\,d\mu_{\infty}(x).\]
\qed

Finally, we find an explicit formula for $a(p,r)$. We also prove the following proposition, expressing $a(p,r)$ as a rational function of $1/p.$
\begin{prop}\label{prop. explicit apn}
	For $r\geq1,$ we have
	\[a(p,r)=\begin{cases}
		\left(1-\frac{1}{p}\right)^{-\frac{r(r-1)}{2}}&\mbox{if $r\leq 3$}\\
		\left(1-\frac{1}{p}\right)^{1-2r}\left(\sum_{\ell=0}^{r-1}\binom{r-1}{\ell}^2\frac{1}{p^{\ell}}-\sum_{\ell=0}^{r-3}\binom{r-3}{\ell}\binom{r+1}{\ell+2}\frac{1}{p^{\ell+1}}\right)&\mbox{if $r\geq3$}
	\end{cases}\]
\end{prop}
\begin{proof}
	The definition of $a(p,r)$ implies:
	\[a(p,r)=\frac{2}{\pi}\int_0^{\pi}\left(\frac{e^{i\theta}-e^{-i\theta}}{2i}\right)^2\frac{1}{\left(1-\frac{e^{i\theta}}{p^{1/2}}\right)^r\left(1-\frac{e^{-i\theta}}{p^{1/2}}\right)^r}\,d\theta.\]
	Then change of variable $z=e^{i\theta}$ implies that,
	\[a(p,r)=-\frac{1}{2}\frac{1}{2\pi i}\oint\limits_{|z|=1}\left(z^2-2+\frac{1}{z^2}\right)\frac{1}{\left(1-\frac{z}{p^{1/2}}\right)^r\left(1-\frac{1}{zp^{1/2}}\right)^r}\frac{\,dz}{z}.\]
	Utilizing the Taylor expansion: 
	\[\frac{1}{(1-z)^r}=\sum\limits_{\ell\geq0}\binom{r+\ell-1}{\ell}z^{\ell},\]
	we obtain:
	\[a(p,r)=-\frac{1}{2}\frac{1}{2\pi i}\oint\limits_{|z|=1}\left(z-\frac{2}{z}+\frac{1}{z^3}\right)\left(\sum_{\ell\geq0}\binom{r+\ell-1}{\ell}\left(\frac{z}{p^{1/2}}\right)^{\ell}\right)\left(\sum_{\ell\geq0}\binom{r+\ell-1}{\ell}\left(\frac{1}{zp^{1/2}}\right)^{\ell}\right)\,dz.\]
	Open the parentheses and apply Cauchy's integral theorem, 
	\[a(p,r)=\sum_{\ell=0}^{\infty}\binom{r+\ell-1}{\ell}^2\frac{1}{p^{\ell}}-\sum_{\ell=0}^{\infty}\binom{r+\ell-1}{\ell}\binom{r+\ell+1}{\ell+2}\frac{1}{p^{\ell+1}}.\] 
	By the definition of hypergeometric function ${}_2F_1$ (\cite[Equation 9.100]{GradshteynRyzhik2007}), we obtain:
	\[\sum_{\ell=0}^{\infty}\binom{r+\ell-1}{\ell}^2\frac{1}{p^{\ell}}={}_2F_1(r,r;1,1/p).\]
	and
	\[\sum_{\ell=0}^{\infty}\binom{r+\ell-1}{\ell}\binom{r+\ell+1}{\ell+2}\frac{1}{p^{\ell+1}}=\frac{r(r+1)}{2p}{}_2F_1(r,r+2;3;1/p).\]
	This yields:
	\[a(p,r)={}_2F_1(r,r;1,1/p)-\frac{r(r+1)}{2p}{}_2F_1(r,r+2;3;1/p).\]
Next, apply the Pfaff transformation (\cite[Equation 9.131(3)]{GradshteynRyzhik2007}):
\[F(\alpha,\beta;\gamma;z)=(1-z)^{\gamma-\alpha-\beta}F(\gamma-\alpha,\gamma-\beta;\gamma;z),\]
and we obtain:
\[a(p,r)=\left(1-\frac{1}{p}\right)^{1-2r}\left\{{}_2F_1(1-r,1-r;1,1/p)-\frac{r(r+1)}{2p}{}_2F_1(3-r,1-r;3;1/p)\right\}.\]
By setting $r=1,2$ and $3$, we can calculate $a(p,r)$ directly. Therefore, we can assume that $r\geq3.$ Then by the definition of the hypergeometric function (\cite[Equation 9.100]{GradshteynRyzhik2007}), we can show
\[{}_2F_1(1-r,1-r;1,1/p)=\sum_{\ell=0}^{r-1}\binom{r-1}{\ell}^2\frac{1}{p^{\ell}}\]
and
\[\frac{r(r+1)}{2p}{}_2F_1(3-r,1-r;3;1/p)=\frac{1}{p}\sum_{\ell=0}^{r-3}\binom{r-3}{\ell}\binom{r+1}{\ell+2}\frac{1}{p^{\ell}}.\]
We conclude that, when $r\geq3,$
\[a(p,r)=\left(1-\frac{1}{p}\right)^{1-2r}\left(\sum_{\ell=0}^{r-1}\binom{n-1}{\ell}^2\frac{1}{p^{\ell}}-\sum_{\ell=0}^{r-3}\binom{r-3}{\ell}\binom{r+1}{\ell+2}\frac{1}{p^{\ell+1}}\right).\]
\end{proof}
\begin{remark}\label{rem. Sato Tate final remark}

By Proposition \ref{prop. explicit apn}, Conjecture \ref{conj. arithmetic Sato-Tate} coincides with Theorem \ref{thm. sato tate, with harmonic weight} when $r=1,2,3.$
    
\end{remark}

\appendix

\section{Remarks on the Weighted Equidistribution Problem: the Cubic Case}\label{apx. sato tate cubic without}
In Proposition \ref{cor. cubic moment fixed ideal without the harmonic weight}, we showed, the main term is dependent on $H(\fp^{\ell})/H(\cO_F),$ where $H(\fp^{\ell})$ is defined in \eqref{cubic moment, main term,  with harnomic weight}. By using Mathematica, we can show
\small
\begin{flalign*}
    \frac{H(\fp^{\ell})}{H(\cO_F)}&=\frac{\ell^3 (N(\fp)-1)^2 (N(\fp)+1)^3}{12 N(\fp) \left(N(\fp)^4+N(\fp)^3+4 N(\fp)^2+N(\fp)+1\right)}\\
    &+\frac{\ell^2 (N(\fp)-1) (N(\fp)+1)^2 \left(5 N(\fp)^2+2 N(\fp)+1\right)}{8 N(\fp) \left(N(\fp)^4+N(\fp)^3+4 N(\fp)^2+N(\fp)+1\right)}\\
    &+\frac{\ell (N(\fp)+1) \left(17 N(\fp)^4+12 N(\fp)^3+20 N(\fp)^2-1\right)}{12 N(\fp) \left(N(\fp)^4+N(\fp)^3+4 N(\fp)^2+N(\fp)+1\right)}\\
     &+\frac{15 N(\fp)^4}{16 \left(N(\fp)^4+N(\fp)^3+4 N(\fp)^2+N(\fp)+1\right)}+\frac{21 N(\fp)^3}{16(N(\fp)^4+N(\fp)^3+4 N(\fp)^2+N(\fp)+1)}+\\
    &+\frac{27 N(\fp)^2}{8 (N(\fp)^4+N(\fp)^3+4 N(\fp)^2+N(\fp)+1)}+\frac{13 N(\fp)}{8(N(\fp)^4+N(\fp)^3+4 N(\fp)^2+N(\fp)+1)}\\
    &+\frac{11}{16 (N(\fp)^4+N(\fp)^3+4 N(\fp)^2+N(\fp)+1)}+\frac{1}{16 N(\fp) (N(\fp)^4+N(\fp)^3+4 N(\fp)^2+N(\fp)+1)}\\
    &+(-1)^\ell \left(\frac{N(\fp)^4}{16 (N(\fp)^4+N(\fp)^3+4 N(\fp)^2+N(\fp)+1)}-\frac{5 N(\fp)^3}{16 (N(\fp)^4+N(\fp)^3+4 N(\fp)^2+N(\fp)+1)}\right.\\
    &\hspace{16mm}+\frac{5 N(\fp)^2}{8 (N(\fp)^4+N(\fp)^3+4 N(\fp)^2+N(\fp)+1)}-\frac{5 N(\fp)}{8 (N(\fp)^4+N(\fp)^3+4 N(\fp)^2+N(\fp)+1)}\\
    &\hspace{16mm}\left.+\frac{5}{16 (N(\fp)^4+N(\fp)^3+4 N(\fp)^2+N(\fp)+1)}-\frac{1}{16 (N(\fp)^5+N(\fp)^4+4 N(\fp)^3+N(\fp)^2+N(\fp))}\right)\\
\end{flalign*}

\normalsize
Then apply Lemma \ref{lem, main lem for sato-tate with harmonic weight} and Lemma \ref{generating series}, we finish the proof for the cubic case in Theorem \ref{thm. sato tate, without harmonic weight}.

\section*{Acknowledgment}
We would like to express our gratitude to Vorrapan Chandee, Xiannan Li, Yongxiao Lin, and Max Wenqiang Xu for their encouragements and valuable suggestions.

\bibliographystyle{alpha}	
\bibliography{ASTC.bib}

\end{document}